\documentclass[
numbers=noenddot,
twoside,10pt,a4paper,
abstract=true
]{scrartcl}

\usepackage[latin1]{inputenc}          % Kodierung einstellen
\usepackage[T1]{fontenc}
\usepackage{graphicx}                
\usepackage{hyperref}
\hypersetup{linktoc=all,colorlinks=true,urlcolor=red!80,linkcolor=green!70!blue!100!,citecolor=green!70!blue!100!}
\usepackage{amsmath}                   % AMS (American Mathematical Society): weitere Fähigkeiten
\usepackage{amssymb}                   % Weitere AMS Symbole
\usepackage{amsfonts}                  % Weitere AMS Zeichensätze
\usepackage{amsthm}
\usepackage[framemethod=tikz]{mdframed}
\usepackage{pgfplots} 
\pgfplotsset{compat=1.11}
\usepackage{tikz}
\usepackage{stmaryrd} %symbols
\usepackage{lmodern}  %font
\usepackage{enumitem}
\usepackage{bbm} %\mathbbm{1} for indicator function
\usepackage{lastpage}
\usepackage{makeidx}
\usepackage[nowatermark]{fixmetodonotes}
\usepackage{epstopdf}
\usepackage{epsfig}
\usepackage[nonumberlist,nopostdot]{glossaries} %to create list of symbols
\usepackage{mathtools} %to use for instance dcases
\usepackage[draft=false,headsepline=true,footsepline=false]{scrlayer-scrpage} %specify header and footer
\usepackage[top=2.5cm,left=3cm,right=3cm,bottom=2.5cm]{geometry}
\usepackage{caption}

\usepackage{mymacros}
\graphicspath{{img/}}
 
\mdfdefinestyle{mystyle}{backgroundcolor=lightgray!0,hidealllines=true,leftmargin=0pt,rightmargin=0pt,innerleftmargin=0pt}

{
	\newtheoremstyle{break}{3pt}{3pt}{\normalfont}{}{\bfseries}{}{\newline}{}
	\theoremstyle{break}
	\newmdtheoremenv[style=mystyle]{thm}{Theorem}[section]
	\newmdtheoremenv[style=mystyle]{thm*}{Theorem}
	\newmdtheoremenv[style=mystyle]{cor}[thm]{Corollary}
	\newmdtheoremenv[style=mystyle]{lem}[thm]{Lemma}
	\newmdtheoremenv[style=mystyle]{prop}[thm]{Proposition}
	\newmdtheoremenv[style=mystyle]{dfn}[thm]{Definition}
	\theoremstyle{plain}
	
	\newtheorem*{ex*}{Example}
	\theoremstyle{remark}
	\newtheorem{rem}[thm]{Remark}
	\newtheorem*{rem*}{Remark}
	\newtheorem{assum}[thm]{Assumption}
}

\newenvironment{bew*}{\vspace{-1.5ex}}{\hspace{1cm}\hspace*{\fill}$\Box$\\%
}

\fontfamily{phv}\selectfont                    % Zur Änderung der Standardschrift in Helvetica
\let\emph\textbf                               % emph bewirkt fett statt kursiv

\numberwithin{equation}{section}
\renewcommand{\theequation}{\arabic{section}.\arabic{equation}}

\newglossarystyle{mysuper}{
	\glossarystyle{super}
	
	\setlength{\glsdescwidth}{0.8\textwidth}
	}

\automark[section]{subsection}	
\clearscrheadfoot 

\lohead{\leftmark}
\rehead{\rightmark}
\ohead{\pagemark}

\providecommand{\keywords}[1]{\textbf{Key words:} #1}
\providecommand{\AMS}[1]{\textbf{AMS subject classifications:} #1}

\KOMAoptions{%
    headsepline=false}

\clearscrheadfoot
\cohead{Compressible multi-component flow in porous media}
\cehead{Lukas Ostrowski and Christian Rohde}
\ohead{\pagemark}

\author{Lukas Ostrowski\thanks{Institute of Applied Analysis and Numerical Simulation, University of Stuttgart
  (\href{mailto:Lukas.Ostrowski@mathematik.uni-stuttgart.de}{Lukas.Ostrowski@mathematik.uni-stuttgart.de).}}
\and Christian Rohde\thanks{Institute of Applied Analysis and Numerical Simulation, University of Stuttgart,  (\href{mailto:Christian.Rohde@mathematik.uni-stuttgart.de}{Christian.Rohde@mathematik.uni-stuttgart.de}).}}

\title{Compressible multi-component flow in porous media with Maxwell-Stefan diffusion}
\date{}

\begin{document}
\maketitle

\begin{abstract}
We introduce a Darcy-scale model to describe compressible  multi-component  flow in a fully saturated porous medium. In order to capture cross-diffusive effects between the different species correctly, we make use of the Maxwell--Stefan theory in a thermodynamically consistent way.\\  For inviscid flow the model  turns out to be  a nonlinear system of hyperbolic balance laws.
We  show  that the  dissipative structure of the Maxwell-Stefan operator  permits to  guarentee    the existence of global classical 
solutions for initial data close to equilibria.  Furthermore it is proven  by relative entropy techniques that solutions of the Darcy-scale model tend in a  certain long-time regime to solutions of a parabolic limit system.
\end{abstract}

\begin{center}
\parbox{0.9\linewidth}{
\keywords{compressible porous media flow, Maxwell--Stefan diffusion, classical wellposedness, relative entropies}

\AMS{35L65,76N10,76S05}
}
\end{center}

\section{Introduction}
Multi-component flows in porous media appear in various fields of applications such as fuel cells, oxygen sensors, and respiratory airways \cite{birgersson2003reduced,MR2777658,Haberman20043617,You2006219}. To highlight the modeling challenge let us focus on the last example. As stated in \cite{MR2777658}, the bronchial tree can be divided into two parts. In the lower part the velocity of the air is very small, such that the dynamics of the gas mixture is mainly dictated by diffusive effects. For the treatment of certain diseases of the lung, a gas mixture (Heliox) is used to improve the patient's well-being. Mathematical models can be used to analyze how to achieve the greatest benefit for the patient. In this situation the classical Fickian diffusion law is too simplistic. 
Important effects, for instance uphill diffusion (\cite{C4CS00440J}), cannot be covered by this approach. By uphill diffusion we mean flux from regions of low concentrations to ones with high concentration, see \cite{Boudin20121427} and references therein. 
Let us additionally note that Duncan \& Toor have given an experimental example of a three-component gas mixture in \cite{AIC:AIC690080112}, which clearly demonstrates the uphill diffusion effect.
A generalization of the Fickian approach is needed which roots in the  classical works of Maxwell \cite{Maxwell1866} and Stefan \cite{Stefan1871}. It has led  to the concept which is nowadays called Maxwell--Stefan diffusion exploiting binary interactions between different species of the mixture. This approach captures more complex diffusive effects, but leads to a coupled nonlinear system of partial differential equations and is therefore mathematically more challenging.

In this paper we provide a mathematical model for compressible multi-component flow in fully saturated porous media on the Darcy-scale. This model takes the form of a nonlinear hyperbolic balance law, therefore classical solutions might fail to exist globally \cite{Dafermos:1315649}. However, we show that dissipative effects of the Maxwell--Stefan diffusion and porous media friction suffice to ensure the classical wellposedness for initial data close to equilibrium. Note that we account for the effect of the solid skeleton in the porous medium like in dusty gas models from  \cite{Mason1983,krishna}. It is regarded as an additional component of the mixture with zero velocity and constant density. 
In contrast to these dusty gas models, which rely on the   kinetic theory of gases, we use the continuum thermodynamics framework as developed in \cite{Bothe2014a}.
The solutions of the resulting system automatically satisfy an entropy condition and hence the second law of thermodynamics. 
If we set the Maxwell--Stefan coupling terms to zero, we obtain a system of uncoupled equations which correspond to compressible Euler equations with friction. This model is used for single-component flow through porous media. It has been shown in \cite{Fang2009244,Hsu2016708,Huang2005,Tzavaras,zeng} that the solutions to this system tend in a long-time limit to the solution of a parabolic porous media equation. We establish a corresponding result for the multi-component case that results in a parabolic system of porous media equations. This system is similar to the multi-component system of \cite{MR3090648}.

The paper is organized as follows. In Section \ref{sec:mod} we derive the governing equations for compressible multi-component flow in porous media exploiting the continuum thermodynamics framework. We start with multiple Euler equations with friction, which are coupled by a right hand side using Maxwell--Stefan cross-diffusion terms.
In the following Section \ref{sec3} we investigate the existence of smooth solutions to this system. The dissipative effects due to Maxwell--Stefan diffusion and friction are shown to fit exactly into the wellposedness theory of \cite{yong} for general hyperbolic balance laws (see Theorem \ref{thmregu}).
The second major goal of the paper on the existence of a parabolic limit system is pursued in Section \ref{sec:limit}. We use a relative entropy framework to prove our main Theorem \ref{thm:paraboliclimit}. For this purpose we adapt techniques from \cite{Tzavaras}, where the convergence of the compressible Euler system with friction to the porous media equation has already been proven. We finally conclude in Section \ref{sec:conclusions}.
 
 %%%%%%%%%%%%%%%%%%%%%%%%%%%%%%%%%%%%%%%%%%%%%%%%%%%%%%%%%%%
\section{Compressible Flow in Porous Media}\label{sec:mod}
We first review  in Section \ref{sec:singlecomp} a single-component model and highlight available analytical results which we will generalize to 
the multi-component case in the remainder of the paper. The multi-component model itself  is derived using fundamental thermodynamical principles in 
Section \ref{sec:multicomp}. It follows the presentation in \cite{Bothe2014a} for free flow problems extending the classical Fickian diffusion modeling to a Maxwell--Stefan approach.
\subsection{Single-Component Flow\label{sec:singlecomp}}
Compressible single-component flow with friction induced by the resistance of the solid skeleton in a porous medium 
can be described  on a macroscopic averaged scale by the Euler--Darcy model, see \cite{Marcati1990} and references therein. We consider the isothermal situation with constant temperature. For $t\in (0,T)$, $T>0$, and  $\vec{x}$ from 
the entire $ \R^d$ 
the unknowns mass density $\rho = \rho(\vec{x},t)>0$
and momentum $\vec{m} = \rho \vec{v} \in \R^d$   with velocity $\vec{v}=\vec{v}(\vec{x},t) \in \R^d$ satisfy in this case  the  system 
\begin{align}
\begin{split}
\label{Euler-Darcy}
 \del_t \rho + \div(\vec{m}) &= 0, \\
\del_t \vec{m}+ \div\left(\frac{\vec{m} \vec{m}^\top}{\rho} +p(\rho)\mathcal{I}_d\right) &= -  M\vec{m}
\end{split} & \text{in } \R^d \times (0,T).
\end{align}
Here $p=p(\rho)$ is the smooth pressure function, $M>0$ is the mobility constant, and $\mathcal{I}_d$ is the $d$-dimensional unit matrix. Note that we use in \eqref{Euler-Darcy} the  same notation for the divergence operator as applied to  vector-  or matrix-valued functions, see 
Appendix  \ref{sec:appA}.\\
It is well-known that \eqref{Euler-Darcy} is a hyperbolic system of nonlinear balance laws as long as the pressure is monotonically increasing.
Shock-type singularities might evolve in finite time regardless of the initial data's regularity.  It is of overall importance that \eqref{Euler-Darcy} 
is endowed with a entropy-entropy flux pair  (see \eqref{pairone} below) which can be used to ensure an appropriate form of the second law of thermodynamics for classical as well as 
weak solutions. Precisely,    
solutions of \eqref{Euler-Darcy} are required to satisfy 
the entropy inequality
\begin{equation*}%\label{eq:ent}
\del_t \eta(\rho,\vec{m}) + \div \vec{q}(\rho,\vec{m}) \leq 0 \quad \text{in } \mathcal{D}'(\R^d\times (0,T))
\end{equation*}
for the entropy-entropy flux pair $(\eta,\vec{q})$ given by
\begin{equation}\label{pairone}
		\eta(\rho,\vec{m}) = \frac{1}{2} \frac{|\vec{m}|^2}{\rho} +\rho \psi(\rho), \qquad 
		\vec{q}(\rho,\vec{m})    =  \frac{1}{2} \vec{m}\frac{|\vec{m}|^2}{\rho^2} + \vec{m} ( \psi(\rho) +  \rho  \psi'(\rho)).
\end{equation}
Here the free energy density $ \rho\psi$ is determined from the Gibbs--Duhem relation
\begin{align}\label{eq:free_energy_density}
\rho\psi(\rho) = (\rho\psi)'(\rho) - p(\rho). 
\end{align}
While the  nonlinear flux in \eqref{Euler-Darcy}  can drive shock waves, the dissipative effect of the friction term   might suffice to counteract the destabilizing effect of the  flux. Depending on the initial data and the size of $M$ the initial value problem  (IVP) for \eqref{Euler-Darcy} 
can have in fact global smooth solutions (see e.g. \cite{Fang2009244}). We will show that a similar result holds for the multi-component case. Furthermore, the
the  dissipative friction effect leads to  certain limit regimes such that  \eqref{Euler-Darcy} changes type in the limit. 
We   consider a long-time and large-mobility regime in \eqref{Euler-Darcy}, i.e., the time $t$ is scaled by a small parameter $\eps > 0$ and the mobility $M$ by $\eps^{-1}$. After rescaling  \eqref{Euler-Darcy} and renaming the variables in an obvious way \eqref{Euler-Darcy} is recasted in the 
 form
\begin{equation}
	\label{Euler-Darcy-eps}
	\begin{array}{rcl}
		\eps \del_t \rho^\eps + \div(\vec{m}^\eps) &= &0, \\
		\eps \del_t(\vec{m}^\eps)+ \div\left(\dfrac{\vec{m}^\eps \vec{m}^{\eps\top}}{\rho^\eps} +p(\rho^\eps)\mathcal{I}_d\right)&=&  - \frac{1}{\eps} M\vec{m}^\eps 
	\end{array}\text{ in } \R^d \times (0,T). 
\end{equation}
In \cite{Tzavaras} (but see also \cite{Huang2005}) it has been shown that the sequence of densities ${\rho}_\eps >0$  solving the  initial value problems  for \eqref{Euler-Darcy-eps} converge for 
$\eps \to 0$ towards a solution $\bar\rho $ of the IVP for the porous medium equation
\begin{equation}\label{pm}
\del_t \bar\rho - M^{-1}\div (\nabla p(\bar\rho)) =0  \text{ in } \R^d \times (0,T).
\end{equation}
In other words the hyperbolic balance laws turn into a parabolic evolution with much  more regular solution behavior.
 We aim at a corresponding result  for the multi-component case (for the parabolic system  \eqref{Darcy-MS-limit} that reduces to \eqref{pm} in the one-component case).

%%%%%%%%%%%%%%%%%%%%%%%%%%%%%%%%%%%%%%%%%%%%%%%%%%
\subsection{Multi-Component Flow\label{sec:multicomp}}
%%%%%%%%%%%%%%%%%%%%%%%%%%%%%%%%%%%%%%%%%%%%%%%%%%
While single-component flow in a porous medium is well understood, much less is known for multi-component flow. 
As long as bulk viscosity is neglected 
  standard model approaches take the form of the 
Euler equations with a damping term in the momentum equations like \eqref{Euler-Darcy}.
However, in the case of porous media and multi-component gaseous mixtures, inter-component viscosity effects become important which do not occur 
in the single-component case.    
The classical Fickian approach does not suffice to describe these  diffusion phenomena. As a possible remedy we favor in this paper  a Maxwell--Stefan ansatz and,
in order to derive governing equations  in a thermodynamically consistent way,  follow the work of Bothe \& Dreyer \cite{Bothe2014a} for free 
multi-component flow.

\subsubsection{Multi-Component Flow and Maxwell--Stefan Diffusion\label{sec:MSfree}}
Let a fluid mixture
 consist of $n\in \N$ components $A_1,...,A_n$ with corresponding mass densities $\rho_i=\rho_i(\vec{x},t)>0$ and velocities $\vec{v}_i = \vec{v}_i(\vec{x},t) \in \R^d,$ $i=1,...,n$. \\
We define the total mass density $\rho$    and the barycentric velocity $\vec{v}$   (not to be interchanged with the single-component case 
in Section \ref{sec:singlecomp}) as
\begin{align*} 
\rho:= \sum_{i=1}^{n}\rho_i, \quad \vec{v}:= \frac{1}{\rho} \sum_{i=1}^{n}\rho_i\vec{v}_i.
\end{align*}
Further, we define the diffusion velocities
\[ \vec{u}_i:= \vec{v}_i - \vec{v} \in \R^d. \]

We ignore mass exchange as well as  exterior forces. Restricting ourselves to the case of a simple mixture, the component pressures $p_i$ depend on $\rho_i$ only, i.e.~they satisfy
 $p_i = p_i(\rho_i)$. 
For $i=1,\dots,n$ we  start then from  the partial balances of mass and momentum given by
\begin{subequations}
\begin{align}
\del_t \rho_i + \div(\rho_i\vec{v}_i)&=0, \label{partialmass} \\ 
\del_t(\rho_i\vec{v}_i) + \div(\rho_i \vec{v}_i \vec{v}_i^\top+p_i(\rho_i)\mathcal{I}_d) &= \vec{f}_i. \label{partialmoment}
\end{align}
\end{subequations}

Here $\vec{f}_i \in \R^d$ states the momentum production due to diffusive mixing, later to be specified with the Maxwell--Stefan ansatz.
As a natural requirement  the conservation law for total momentum has to hold, which implies the condition
\begin{align*}
%\label{sumoverf}
\sum_{i=1}^{n} \vec{f}_i = 0. 
\end{align*} 
The crucial part is now to find an expression for $\vec{f}_i$ such that %
with the (physical) entropy production $\zeta$  (see \eqref{zeta} below) 
the second law of thermodynamics holds true. 

We introduce for each component $A_i$  a strictly convex free energy density $ h_i(\rho_i)= \rho_i\psi_i(\rho_i)$ that relates to the partial pressure $p_i(\rho_i)$ via the Gibbs--Duhem equations (see \eqref{eq:free_energy_density} for the single-component velocity)
$$ h_i(\rho_i)+p_i(\rho_i)=\rho_i h_i'(\rho_i).$$

Thus, the strict convexity of $\rho_i\psi_i$ implies 
\begin{align*}
	%\label{pmonoton}
	p_i^\prime(\rho_i) > 0.
\end{align*} Moreover,  the function 
\begin{align} \label{simple2}
h(\rho_1,\dots,\rho_n) := \sum_{i=1}^n h_i(\rho_i)
\end{align}
acts as mixture free energy for simple mixtures.

For the special case of simple isothermal inviscid fluid mixtures without chemical reactions the entropy production  $\zeta$
of some solution $(\rho_1,\ldots,\rho_n, \vec{m}_1^\top,\ldots,\vec{m}_n^\top)$ of \eqref{partialmass}, \eqref{partialmoment}
is derived in \cite{Bothe2014a} and reads as
\begin{align}\label{zeta}
\zeta = -\sum_{i=1}^n \vec{u}_i \cdot \vec{f}_i.
\end{align}

The nonnegativity of the entropy production $\zeta$  leads  then to the requirement
\begin{align}
\label{entropy2}
-\sum_{i=1}^{n-1}(\vec{u}_i-\vec{u}_n)\cdot\vec{f}_i \geq 0.
\end{align}
In the following we make the Maxwell--Stefan ansatz for $\vec{f}_i$ to guarantee that \eqref{entropy2} holds true. Let 
\begin{align}\label{eq:defTtilde}
\tilde{\mathcal{T}}:=(\tau_{ij})_{i,j=1}^{n-1}\succ 0, \text{ with } \tau_{ij} = \tau_{ij}(\rho_i,\rho_j)
\end{align}
be a positive-definite matrix.
With \eqref{eq:defTtilde} we set
\begin{align}
\label{msansatz}
\vec{f}_i = - \sum_{j=1}^{n-1}\tau_{ij}(\vec{u}_j-\vec{u}_n), \quad i=1,...,n-1.
\end{align}

In order to make the right hand side in \eqref{msansatz} symmetric regarding the components, we extend 
$\tilde{\mathcal{T}} \in \R^{(n-1) \times (n-1)}$ to the \textit{Maxwell--Stefan matrix} $\mathcal{T} := (\tau_{ij})_{i,j=1}^n \in \R^{n \times n}$ \cite{taylorkrishna} by
\begin{align}\label{extendedmax}
\tau_{nj}= -\sum_{i=1}^{n-1}\tau_{ij}, \quad j=1,...,n-1, \qquad \tau_{in}=-\sum_{j=1}^{n-1}\tau_{ij}, \quad i=1,...,n.
\end{align}
Additionally let 
\begin{align}\label{tauleq0}
\tau_{ij}\leq 0 \text{ for all } i\neq j .
\end{align} 

In the case of binary interactions the matrix $\mathcal{T}$ is symmetric and can be shown to be positive semi-definite provided \eqref{eq:defTtilde} holds, see \cite{Bothe2014a}.

The following ansatz for the components of $\tilde{\mathcal T}$  is made to match the requirements \eqref{extendedmax}, \eqref{tauleq0}. 
For some numbers  $\lambda_{ij} \ge 0  $ with  $\lambda_{ij}(\rho_i,\rho_j)=\lambda_{ji}(\rho_j,\rho_i)$ for $i\neq j$  
we define 
\begin{align}\label{lambda}
\tau_{ij}=-\lambda_{ij}(\rho_i,\rho_j)\rho_i\rho_j \qquad (i\neq j).  %, \text{ with } \lambda_{ij}(\rho_i,\rho_j)=\lambda_{ji}(\rho_j,\rho_i), \lambda_{ij}\geq 0, \text{ for } i\neq j.
\end{align} 
With $\lambda_{ii}= - \sum_{j=1, j\neq i}^n \lambda_{ij}\frac{\rho_j}{\rho_i}$  we  introduce the  negative  semi-definite matrix
\begin{equation}\label{def:lambda}
\Lambda = \Lambda (\vec{r}) =  {\Big(\lambda_{ij}(\rho_i,\rho_j) \Big)}_{i,j=1}^n\in \R^{n\times n}.   
\end{equation}
With the definitions \eqref{extendedmax} we infer from \eqref{msansatz} the relations
\begin{subequations}\label{2.17}
\begin{align}
\vec{f}_i &= - \sum_{j=1}^{n} \tau_{ij}(\vec{u}_j-\vec{u}_n), \quad i=1,...,n, \\
\sum_{j=1}^{n} \tau_{ij} &=0 ,  \hspace*{8.4em}i=1,...,n. \label{lambdasum}
\end{align}
\end{subequations}

Thus, by replacing $\vec{u}_n$ with $\vec{u}_i$ in \eqref{2.17}, we obtain a symmetrical version of \eqref{msansatz}, namely
\begin{align}\label{2.8}
\vec{f}_i = \sum_{j=1}^{n}\tau_{ij}(\vec{u}_i-\vec{u}_j),\quad i=1,...,n.
\end{align}

With \eqref{2.8} and the symmetry of $\mathcal{T}$ the entropy production \eqref{zeta} can be written as
\begin{align*}
\zeta = -\sum_{i=1}^{n} \vec{u}_i \cdot \vec{f}_i= -\frac{1}{2} \sum_{i,j=1}^{n} \tau_{ij}|\vec{u}_i-\vec{u}_j|^2.
\end{align*}
Obviously, condition \eqref{tauleq0} is necessary to achieve $\zeta \geq 0$.
Due to \eqref{lambda} the entropy production reads as
\begin{align}
\label{zetamitlambda}
\zeta = \frac{1}{2}\sum_{i,j=1}^n \lambda_{ij}(\rho_i,\rho_j)  \rho_i\rho_j |\vec{v}_i-\vec{v}_j|^2.
\end{align}

Finally, with the Maxwell--Stefan ansatz the constitutive law for the momentum production $\vec{f}_i$ results from \eqref{lambda} and \eqref{2.8} as
\begin{align}\label{fi}
\vec{f}_i = -  \sum_{j=1}^{n}\lambda_{ij}(\rho_i,\rho_j)\rho_i\rho_j(\vec{v}_i-\vec{v}_j).
\end{align}
Note that in \eqref{fi} the diffusion velocities are replaced by the velocities of the corresponding component.

With this result the partial momentum balances \eqref{partialmoment} read as
\begin{align}\label{eq:partialmoment_MS}
\del_t(\rho_i\vec{v}_i)+\div(\rho_i\vec{v}_i \vec{v}_i^\top+p_i(\rho_i)\mathcal{I}_d)= -\sum_{j=1}^{n}\lambda_{ij}(\rho_i,\rho_j)\rho_i\rho_j(\vec{v}_i-\vec{v}_j).
\end{align}

\subsubsection{Compressible Multi-Component Flow in Porous Media}

So far we considered a free flow problem. 
We realize the porous medium matrix as a static component $A_{\mathrm{pm}}$ of the mixture with velocity $\vec{v}_{\mathrm{pm}}=0$ and density $\rho_{\mathrm{pm}}=\mathrm{const.}$\\
The component $A_{\mathrm{pm}}$ needs no equations for the mass  and momentum balance. However, we have to take into account the effects on the other components.
Hence, the sum from \eqref{fi}
extends to
\[ -\sum_{j=1}^{n}\lambda_{ij}(\rho_i,\rho_j)\rho_i\rho_j(\vec{v}_i-\vec{v}_j)-\lambda_{i,\mathrm{pm}}(\rho_i,\rho_{\mathrm{pm}})\rho_{\mathrm{pm}}\rho_i(\vec{v}_{i}-\vec{v}_{\mathrm{pm}}). \]
In the sequel we ignore the explicit dependence of $\lambda_{i,\text{pm}}$ on the component densities $\rho_i$ and proceed with the mobility constants 
\[ M_i=M_i(\rho_{\mathrm{pm}})\coloneqq \lambda_{i,\mathrm{pm}}(\rho_{\mathrm{pm}})\rho_{\mathrm{pm}}. \]
Then we arrive at our final system which reduces in the single-component case to \eqref{Euler-Darcy}.\\
Define the density vector  $\vec{r} = (\rho_1,\ldots, \rho_n)^\top$ and with $\vec{m}_i = \rho_i\vec{v}_i$ for $i= 1,\ldots,n$ the momentum vector
 $\vec{m} = (\vec{m}_1^\top,\ldots, \vec{m}_n^\top )^\top  $.
We search for  the function   $U =  ( \vec{r}^\top,\vec{m}^\top)^\top$
with values in the state space 
\begin{align}\label{eq:statespace}
G = \R_+^n \times \R^{nd},
\end{align}  that
satisfies the multi-component Euler--Darcy system with Maxwell--Stefan type diffusion 
\begin{align}%\tag{Euler-MS}
\begin{split}
\label{Euler-MS}
\del_t \rho_i &+ \div(\vec{m}_i) = 0, \\
\del_t(\vec{m}_i)&+ \div\left(\frac{\vec{m}_i \vec{m}_i^\top}{\rho_i}+p_i(\rho_i)\mathcal{I}_d\right)  
 \\
 &\qquad = - M_i\vec{m}_i -  \sum_{j=1}^{n} \lambda_{ij}(\rho_i,\rho_j)(\rho_j\vec{m}_i-\rho_i\vec{m}_j)\end{split} & \text{ in } \R^d \times (0,T),
\end{align}
subject to the   initial condition 
\begin{align}\label{Euler-MS_ini}
U(\cdot,0) = U_0 := (\vec{r}_0^\top,\vec{m}_0^\top)^\top= (\rho_{1,0},\dots,\rho_{n,0},\vec{m}_{1,0}^\top,\dots,\vec{m}_{n,0}^\top)^\top
 & \text{ in } \R^d.
\end{align}
Due to the arguments from Section \ref{sec:MSfree} we observe 
that the second law of thermodynamics is automatically satisfied along smooth solution trajectories of \eqref{Euler-MS}, \eqref{Euler-MS_ini}. 
\begin{rem}\label{rem:model} 
\begin{enumerate}

\item When including the porous medium, the condition \eqref{lambdasum} contains the summand of the porous medium part as well. We neglect this fact. That means precisely that we neglect the effect of the porous medium on the conservation of momentum \eqref{fi}. This is in accordance with the single-component case. The effect of the porous medium has the character of a body force. 
\item It turns out in experiments that the $\lambda_{ij}$ are only weakly dependent on the mixture. Often affine functions suffice to describe this dependence, see \cite{taylorkrishna}.
\item The structure of the porous medium is only captured in the scalar parameter $M_i$. If the porous medium is not homogeneous and isotropic, one should allow for matrix-valued parameters with spatial dependence.
\item  The terms $M_i$ scale with  the density of the porous medium, which is significantly larger than the densities of a gaseous mixture.  Hence, typically it holds $M_i \gg \lambda_{ij}$.
\end{enumerate}
\end{rem}

Similar like in Section \ref{sec:singlecomp} we consider a long-time/large-mobility/large-diffusion regime for \eqref{Euler-MS}.
	To be precise, let  $\bar{x},\, \bar{t},\, \bar{\rho}>0,\, \bar{v},\, \bar{p},\, \bar{M},\,$ and $\bar{\lambda}$ be the characteristic scales of the corresponding quantities.  The long-time/large-mobility/large-diffusion  regime is now obtained  from
	\[\frac{\bar{x}}{\bar{v}\bar{t}}=O(\eps), \ \frac{\bar{p}}{\bar{v}^2\bar{\rho}}=O(1), \  \frac{\bar{M}\bar{x}}{\bar{v}}= O(\eps^{-1})\  \text{ and } \frac{\bar{\lambda}\bar{\rho}\bar{x}}{\bar{v}}= O(\eps^{-1}),
	\] using 
	 a small parameter $\eps>0$.
After rescaling  the system $\eqref{Euler-MS}$    and renaming the unknown as 
$U^\eps=(\vec{r}^{\eps\top}=(\rho_1^\eps,\ldots,\rho_n^\eps), \vec{m}^{\eps\top}=(\vec{m}_1^{\eps\top},\ldots, \vec{m}_n^{\eps\top} )    )^\top$   it reads in this regime as
\begin{align}
\label{Euler-Darcy-MS-eps}
\begin{split}
\eps \del_t \rho_i^\eps &+ \div(\vec{m}_i^\eps) = 0, \\
\eps \del_t\vec{m}_i^\eps &+ \div\left(\frac{\vec{m}_i^\eps \vec{m}_i^{\eps,\top}}{\rho_i^\eps} +p_i(\rho_i^\eps)\mathcal{I}_d\right) \\
& \qquad = - \frac{1}{\eps} M_i\vec{m}_i^\eps - \frac{1}{\eps} \sum_{j=1}^{n} \lambda_{ij}(\rho_j^\eps\vec{m}_i^\eps-\rho_i^\eps\vec{m}_j^\eps)
\end{split} & \text{in } \R^d \times (0,T),
\end{align}
with the ($\eps$-dependent) initial conditions 
\begin{align}\label{Euler-Darcy-MS-eps_ini}
\begin{split}  U^\eps(\cdot,0) &= U^\eps_0 := ( {\vec{r}^{\eps\top}_0},\vec{m}^{\eps\top}_0)^\top= (\rho^\eps_{1,0},\dots,\rho^\eps_{n,0},\vec{m}^{\eps\top}_{1,0},\dots,\vec{m}^{\eps\top}_{n,0})^\top
\end{split} & \text{ in } \R^d.
\end{align}
As for the single-component case we will show that the multi-component case admits global smooth solutions exploiting the dissipative effect due to 
friction \textit{and}  the Maxwell--Stefan diffusion. 
The other  major goal  is to prove  that the density component sequence ${\{\vec{r}^\eps\}}_{\eps >0}$  of  solutions of  the IVP \eqref{Euler-Darcy-MS-eps}, \eqref{Euler-Darcy-MS-eps_ini} converges for $\eps \to 0$  to the vector-valued density field  $\vec{\bar{r}}=(\bar{\rho}_1,...,\bar{\rho}_n)^\top  $ solving 
the system of  porous medium equations
\begin{align}
\label{Darcy-MS-limit}
\begin{split}
\del_t  \bar{\vec{r}} - \div \big( (\mathcal{B} (\vec{\bar{r}}) )^{-1}  \nabla\vec{p}( \vec{\bar{r}}) \big) &=0  \text{ in } \R^d \times (0,T),
\end{split}
\end{align}
subject to the initial conditions 
\begin{equation}
 \label{Darcy-MS-limit_ini}   \bar{\vec{r}}(\cdot,0) =  \bar{\vec{r}}_0\text{ in } \R^d.
\end{equation} 
In \eqref{Darcy-MS-limit}  
we used the vector-valued pressure
\begin{equation} \vec{p}(\vec{\bar{r}})=(p_1(\bar{\rho}_1),...,p_n(\bar{\rho}_n))^\top, \label{defvecp}
\end{equation}
and the matrices 
\begin{equation}
\begin{array}{rcl}
\mathcal{B}(\vec{\bar{r}})&=&\mathcal{\tilde{B}}(\vec{\bar{r}})  \otimes \mathcal{I}_d  \in \R^{nd\times nd},\\ [1.2ex]
  \mathcal{\tilde{B}}(\vec{\bar{r}})                         & =& \diag(M_i)-\diag(\bar{\rho}_i)\Lambda \in \R^{n\times n}. %,\\ [1.2ex]
%\Lambda &=& \left(\lambda_{ij}\right)_{i,j=1}^n\in \R^{n\times n}.                        
\end{array} \label{eq:defB}
\end{equation}
The matrix $\Lambda$ from \eqref{def:lambda} is negative semi-definite. This implies  that $ \mathcal{\tilde{B}}$  is positive definite
due  to $M_i>0$ and the positivity of the densities. In particular $\mathcal{B}$ turns then out to be positive definite as the Kronecker product of two
positive definite matrices. 
For 
the definition of the  generalized gradient/divergence operators in the 
system \eqref{Darcy-MS-limit}  and the Kronecker matrix product $\otimes$ in \eqref{eq:defB}   we refer to Appendix \ref{sec:appA}.

\begin{rem}
\begin{enumerate}	
\item	For the single component case $n=1$ the system \eqref{Darcy-MS-limit} reduces to the porous media equation \eqref{pm}.
\item If no porous medium is present, that is $M_i=0$, the system \eqref{Darcy-MS-limit} in this framework corresponds for perfect gas laws
  to the following version of the Maxwell--Stefan equations formulated for the molar concentrations $c_i$ often seen in the literature, e.g.~in\cite{MR3090648}: 
\begin{align}\label{eq:Maxwell--Stefan_con}
\begin{split}
\partial_t c_i + \div \vec{J}_i &= 0,\\
\nabla c_i &= - \sum_{j=1, j\neq i}^n \frac{c_j\vec{J}_i-c_i\vec{J}_j}{D_{ij}}.
\end{split}
\end{align}
\end{enumerate}
Here $D_{ij} = \dfrac{R}{c \mathcal{M}_i  \mathcal{M}_j \lambda_{ij}}$, with the ideal gas constant $R$, total molar concentration $c= \sum_{i=1}^n c_i$ and molar masses $\mathcal{M}_i$.
\end{rem}

%%%%%%
\section{Existence of Smooth Solutions in Multiple Space Dimensions}\label{sec3}
%%%%%
The main result in this section is Theorem \ref{thmregu} on the classical wellposedness of the IVP for system \eqref{Euler-MS}. 
To this end  we  propose an entropy concept for \eqref{Euler-MS} and  adapt a result of Yong \cite{yong} on hyperbolic balance laws, see Appendix \ref{sec:appB}. 
It exploits dissipative effects of the  balance terms  that counteract  the development of   singularities driven by the 
hyperbolic flux \cite{Dafermos:1315649}. 
To state and prove our  main result Theorem \ref{thmregu}  below we  summarize all assumptions on the system \eqref{Euler-MS} 
according to the notations from Section \ref{sec:mod}.

\begin{assum}\label{assumption}
\begin{itemize}
\item[(i)] The functions  $\lambda_{ij} \in    C^\infty ((0,\infty)^2, \R), \ i,j=1,\dots,n$, satisfy  \eqref{lambda} and \eqref{lambdasum}.
\item[(ii)] The symmetric matrix $\Lambda(\vec{r})  =  {\Big(\lambda_{ij}(\rho_i,\rho_j) \Big)}_{i,j=1}^n$ is  negative semi-definite for all $\vec{r} \in (0,\infty)^n$. 
\item[(iii)] The free energy densities  $ h_i = \rho_i\psi_i \in C^3((0,\infty))$ are strictly convex for $i=1,\dots,n$.
\item[(iv)] The mobility constants $M_i$  are positive for $i=1,\ldots,n$.
\end{itemize}
\end{assum}

Theorem \ref{yongtheorem} applies to general hyperbolic balance laws. With
$U = ( \rho_1,\dots,\rho_n, ( \vec{m}_1^\top,\dots,\vec{m}_n^\top))^\top$  we can rewrite  \eqref{Euler-MS} in this form, that is 
\begin{equation}\label{genericFormc}
\del_tU + \sum_{\alpha=1}^d \del_{x_\alpha}F_\alpha(U)= S(U) = \begin{pmatrix} \vec{0}\\ \vec{s}(\vec{r},\vec{m}) \end{pmatrix}.
\end{equation}
The fluxes  $F_\alpha(U) \in \R^{(d+1)n}$  and  the source  $ \vec{s}(\vec{r},\vec{m})\in \R^{nd} $ are given by
\begin{equation} 
\begin{array}{rcl}%\label{def:s}
F_\alpha(U) &\!\!\!=& \!\!\!\left(
\vec{m}_1^{(\alpha)},  \dots , \vec{m}_n^{(\alpha)} ,
\frac{\vec{m}_1^{(\alpha)}}{\rho_1}\vec{m}_1^\top + p_1(\rho_1)\vec{e}_\alpha^\top,\dots,\frac{\vec{m}_n^{(\alpha)}}{\rho_n}\vec{m}_n^\top + p_n(\rho_n)\vec{e}_\alpha^\top
\right)^\top \\[3ex]
\vec{s}(\vec{r},\vec{m})&\!\!\!=&\!\!\!
\begin{pmatrix}
-M_1\vec{m}_1- \sum_{j=1}^{n}\lambda_{1j}(\rho_j\vec{m}_1-\rho_1\vec{m}_j) \\
-M_2\vec{m}_2- \sum_{j=1}^{n}\lambda_{2j}(\rho_j\vec{m}_2-\rho_2\vec{m}_j) \\
\vdots \\
-M_n\vec{m}_n- \sum_{j=1}^{n}\lambda_{nj}(\rho_j\vec{m}_n-\rho_n\vec{m}_j) 
\end{pmatrix}.
\end{array}\label{defF} 
\end{equation}
Here we used $\vec{m}_i=(\vec{m_i}^{(1)},\ldots,\vec{m_i}^{(d)} )^\top$ and $\vec{e}_\alpha$ denotes the $\alpha$-th unit vector.
Furthermore, an entropy-entropy flux pair $(\eta,\vec{q}) \in C^2(G)$  for \eqref{genericFormc} on the state space $G$ from  \eqref{eq:statespace} is required. Following \cite{Dafermos:1315649} the tuple
$(\eta,\vec{q})$ is  called an entropy-entropy flux pair to the system \eqref{Euler-MS}
provided $\DD^2\eta(U)$ is positive-definite  and  the compatibility conditions 
\begin{align} \label{eq:compat}
\DD \eta(U) \DD F_\alpha(U) = \DD q_\alpha(U), \quad \alpha = 1,\dots,d,
\end{align}
are satisfied for all $U\in G$. Motivated by the considerations in Section \ref{sec:mod} we suggest for \eqref{Euler-MS} the functions
\begin{equation}\label{eta}
		\eta(U) = \frac{1}{2} \sum_{i=1}^n \frac{|\vec{m}_i|^2}{\rho_i} + \sum_{i=1}^n  h_i(\rho_i), \qquad 
		\vec{q}(U)    =  \frac{1}{2} \sum_{i=1}^n \vec{m}_i\frac{|\vec{m}_i|^2}{\rho_i^2} + \vec{m}_i h_i(\rho_i)'.
\end{equation}
Note that $\eta$ in \eqref{eta} is obviously strictly convex, by Assumption \ref{assumption} (iii).

\begin{thm}[Global classical wellposedness of the IVP for \eqref{Euler-MS}]
	\label{thmregu}
  Let  $s\geq \lfloor d/2 \rfloor + 2$ and let  Assumption \ref{assumption} hold.
   Consider a static equilibrium solution $\hat{U} \in G$ of  \eqref{Euler-MS} of the form 
	\begin{align} \label{eqstate}
	 \hat{U}= (\hat{\rho}_1,...,\hat{\rho}_n, \mathbf{0},...,\mathbf{0})^\top, \quad \hat{\rho}_i > 0, \quad i=1,...,n.
	 \end{align}
	Then there exists a constant $c_1 > 0,$ such that    for all  $U_0 \in H^s(\R^d) $   with 
	 \[ \|U_0-\hat{U}\|_{H^s} \leq c_1   %\text{ and } U_0(\vec{x}) \in G   \text{ for all } \vec{x}\in \R^d  
	 \] 
	 and all $T>0 $ the IVP  \eqref{Euler-MS}, \eqref{Euler-MS_ini}  has a unique solution 
	 $U \in C([0,T),H^s(\R^d))$ taking values in the state space $G$. 
	
Additionally the solution $U$ satisfies the entropy inequality
\begin{equation}  \label{entropyineq}\del_t\eta(U) + \div \vec{q}(U) \leq -\zeta - \sum_{i=1}^n M_i\frac{|\vec{m}_i|^2}{\rho_i},
\end{equation}

with $\zeta=  \sum_{i,j=1}^n \frac{\lambda_{ij}  (\rho_i,\rho_j)}{2\rho_i\rho_j}|\rho_j \vec{m}_i-\rho_i  \vec{m}_j|^2$ (see \eqref{zetamitlambda}).
\end{thm}

\begin{proof}
For better readability we omit the argument $(\rho_i,\rho_j)$ of the functions $\lambda_{ij}$.
With \eqref{genericFormc}  we meet the setting \eqref{yong2} to apply Theorem  \ref{yongtheorem}. 
We compute  for the equilibrium state $\hat{U}$ from \eqref{eqstate} the  Jacobian
\[ \DD_\vec{m} \vec{s}(\hat{U}) =  \begin{pmatrix}
-M_1- \sum\limits_{j=1, j\neq 1}^{n} \hat{\rho}_j\lambda_{1j}& \hat{\rho}_1\lambda_{12},& \dots& \hat{\rho}_1\lambda_{1n} \\
\hat{\rho}_2\lambda_{21}& \ddots & & \vdots \\ 
\vdots& & \ddots  &\hat{\rho}_{n-1}\lambda_{n-1n-1}\\
\hat{\rho}_n\lambda_{n1}& \dots&\hat{\rho}_n\lambda_{nn-1}& -M_n-\sum\limits_{j=1, j\neq n}^{n}\hat{\rho}_j\lambda_{nj}
\end{pmatrix}\otimes \mathcal{I}_d. \]

From  Assumption \ref{assumption} (i) we have $\hat{\rho}_i\lambda_{ii} = -\sum_{j=1, j\neq i}^{n} \hat{\rho}_j\lambda_{ij}$ which implies
\[ \DD_\vec{m} \vec{s}(\hat{U})= (-\mathcal{M} +  \mathcal{R} \Lambda)\otimes \mathcal{I}_d,\]
with  $\mathcal{M}= \diag(M_i) \in \R^{n\times n}$, $ \mathcal{R}=\diag(\hat{\rho}_i) \in \R^{n\times n}$ being positive-definite. 
The Jacobian $\DD_\vec{m} \vec{s}(\hat{U}) \prec 0$ is in particular  a regular matrix which implies 
condition 1.~in Theorem \ref{yongtheorem}.

For the entropy $\eta$ from \eqref{eta}  and any open set $\hat{\mathcal G} \subset G$ containing $\hat U$  we have $\DD \eta(U) \in \R^{n(d+1)}, \DD^2\eta \in \R^{n(d+1)\times n(d+1)}$, $U\in \hat{\mathcal G}$. The derivatives of $\eta$ read with \eqref{simple2} as
\begin{align*}
\DD \eta(U) &= \left(
\dfrac{\del (\rho_1\psi_1)}{\del\rho_1}(\rho_1) -\dfrac{1}{2}\dfrac{|\vec{m}_1|^2}{\rho_1^2},
\dots, 
\dfrac{\del (\rho_n\psi_n)}{\del\rho_n}(\rho_n) - \dfrac{1}{2} \dfrac{|\vec{m}_n|^2}{\rho_n^2} , 
\dfrac{\vec{m}_1^\top}{\rho_1} ,
\dots,
\dfrac{\vec{m}_n^\top}{\rho_n}
\right)^\top,  \\
\DD \eta(\hat{U}) &= \left(
\dfrac{\del (\rho_1\psi_1)}{\del\rho_1}(\hat{\rho}_1),
\dots,
\dfrac{\del (\rho_n\psi_n)}{\del\rho_n}(\hat{\rho}_n),
\vec{0}^\top,
\dots,
\vec{0}^\top\right)^\top, \\
\DD^2 \eta(U)&= \begin{pmatrix}
\diag\left(\dfrac{1}{\rho_i}p_i'(\rho_i)+\dfrac{|\vec{m}_i|^2}{\rho_i^3}\right) & \blockdiag\left(-\dfrac{\vec{m}_i}{\rho_i^2}\right)^\top\\ 
\blockdiag\left(-\dfrac{\vec{m}_i}{\rho_i^2}\right)& \diag\left(\dfrac{1}{\rho_i}\otimes 1_d\right)
\end{pmatrix}.
\end{align*}
Fro the definition of the operator $\blockdiag$  and the Kronecker product $\otimes$ we refer to the Appendix \ref{sec:appB}.
From the definition of the fluxes $F_\alpha$ in \eqref{defF} we compute 
\begin{align*}  
\DD F_{\alpha}(U) &= 
\left(\begin{smallmatrix}
0_{n\times n} & \mathcal{I}_n \otimes \vec{e}_\alpha^\top \\
\blockdiag\left(-\frac{\vec{m}_i\vec{m}_i^{(\alpha)}}{\rho_i^2}+p_i'(\rho_i)\vec{e}_\alpha\right) & \diag\left(\frac{\vec{m}_i^{(\alpha)}}{\rho_i}1_d\right)+\blockdiag\left(\vec{e_\alpha}\odot \frac{\vec{m}_i}{\rho_i}\right)
\end{smallmatrix}\right),
\end{align*}
where $\vec{e}_\alpha$ denotes the $\alpha$-th unit vector and $1_d := (1,\dots,1)^\top \in \R^d$.
We see that the matrix
\begin{align*}
\DD^2\eta(U)\DD F_{\alpha}(U) &=\left(\begin{smallmatrix}
\diag\left(\dfrac{|\vec{m}_i|^2\vec{m}_i^{(\alpha)}}{\rho_i^4}-\dfrac{p_i^\prime(\rho_i)\vec{m}_i^{(\alpha)}}{\rho_i^2}\right) & \blockdiag\left(\dfrac{p_i^\prime(\rho_i)}{\rho_i}\vec{e}_\alpha-\dfrac{\vec{m}_i\vec{m}_i^{(\alpha)}}{\rho_i^3}\right)^\top \\
\blockdiag\left(\dfrac{p_i^\prime(\rho_i)}{\rho_i}\vec{e}_\alpha-\dfrac{\vec{m}_i\vec{m}_i^{(\alpha)}}{\rho_i^3}\right) & \diag\left(\dfrac{\vec{m}_i^{(\alpha)}}{\rho_i^2}1_d,\dots,\dfrac{\vec{m}_n^{(\alpha)}}{\rho_n^2}1_d\right)
\end{smallmatrix}\right)
\end{align*}
is symmetric. 
Altogether we have verified condition 2.~in Theorem \ref{yongtheorem}, up to now for any open set $\hat{\mathcal G}$ that contains  the equilibrium state 
$\hat U$.\\
To check the third condition \ref{yongtheorem} let $U,\hat{U} \in \hat{\mathcal{G}}$. We use the symmetry of $\Lambda$ and \eqref{zetamitlambda} to obtain
\begin{equation}\label{reform}
\begin{array}{rcl}
-(\DD \eta(U)-\DD \eta(\hat{U}))\cdot S(U) &=& \displaystyle \sum_{i=1}^{n}M_i\frac{|\vec{m}_i|^2}{\rho_i} +  \sum_{j=1}^{n} \lambda_{ij}\left (\frac{\rho_j}{\rho_i} |\vec{m}_i|^2 -\vec{m}_i\cdot\vec{m}_j\right) \\
%&=\sum_{i=1}^{n}M_i\frac{m_i^2}{\rho_i}  + \frac{1}{2} \sum_{i,j=1}^{n} \lambda_{ij} \left(\frac{\rho_j}{\rho_i}m_i^2-m_im_j\right) \\
%&\qquad+\frac{1}{2} \sum_{j,i=1}^{n} \lambda_{ji} \left(\frac{\rho_i}{\rho_j}m_j^2-m_im_j\right) \\ 
%&= \sum_{i=1}^{n}M_i\rho_i v_i^2 + \frac{1}{2} \sum_{i,j=1}^{n} \lambda_{ij} \rho_i\rho_j (v_i^2 -2v_iv_j+v_j^2) \\
&=&\displaystyle\sum_{i=1}^{n}M_i\frac{|\vec{m}_i|^2}{\rho_i} + \frac{1}{2} \sum_{i,j=1}^{n} \lambda_{ij} \rho_i\rho_j \left             |\frac{\vec{m}_i}{\rho_i}-\frac{\vec{m}_j}{\rho_j}\right|^2 \\
&= &\displaystyle\sum_{i=1}^{n}M_i\frac{|\vec{m}_i|^2}{\rho_i} + \zeta.
\end{array}
\end{equation}
Furthermore we now choose $\mathcal G$ as a compact, convex subset of $\hat{\mathcal{G}}$ such that we have for all $U\in \mathcal G$ the estimate
\begin{align*} |S(U)|^2 & %\overset{\phantom{\text{Cauchy--Schwarz}}}{=}
= \sum_{i=1}^{n} \bigg|M_i\vec{m}_i+  \sum_{j=1}^{n} \lambda_{ij} (\rho_j\vec{m}_i-\rho_i\vec{m}_j)\bigg|^2 \\
&%\overset{\text{Young \ref{ineq:young}}\phantom{as1jd}}{\leq} 
\leq \sum_{i=1}^{n} 2M_i^2|\vec{m}_i|^2+ 2\bigg|\sum_{j=1}^{n} \lambda_{ij} (\rho_j\vec{m}_i-\rho_i\vec{m}_j)\bigg|^2 \\
&%\overset{\text{Cauchy--Schwarz}}{\leq} 
\leq \sum_{i=1}^{n} 2M_i^2|\vec{m}_i|^2+ 2n^2\sum_{i,j=1}^{n}  \lambda_{ij}^2 |\rho_j\vec{m}_i-\rho_i\vec{m}_j|^2 \\
&%\overset{\phantom{\text{Cauchy--Schwarz}}}{\leq} 
\leq 2\hat{c}_\mathcal{G} \sum_{i=1}^{n} M_i\frac{|\vec{m}_i|^2}{\rho_i}+ 2\tilde{c}_\mathcal{G} \sum_{i,j=1}^{n} \lambda_{ij}\rho_i\rho_j \left|\frac{\vec{m}_i}{\rho_i}-\frac{\vec{m}_j}{\rho_j}\right|^2, \\
\intertext{with} 
\hat{c}_\mathcal{G} &=\left(\max\limits_{i=1,...,n} M_i\right) \cdot \left(\max\limits_{i=1,...,n}\max\limits_{\rho_i \in\mathcal{G}}\rho_i\right),\\ 
\tilde{c}_\mathcal{G}&= n^2\left(\max\limits_{i,j=1,...,n}\max\limits_{\rho_i,\rho_j \in \mathcal{G}}\lambda_{ij}(\rho_i,\rho_j)\right) \cdot\left(\max\limits_{i=1,...,n}\max\limits_{\rho_i\in \mathcal{G}}\rho_i\right)^2.
%\\ 
%&\leq c_\mathcal{G}^{-1} \left(\sum_{i=1}^{n} M_i\frac{|\vec{m}_i|^2}{\rho_i}+ \zeta\right), \quad c_\mathcal{G}^{-1} = 2\max\left\{{\hat{c}_\mathcal{G}},{\tilde{c}_\mathcal{G}}\right\}.
\end{align*} 
Hence, we get from \eqref{reform} with $c_\mathcal{G}^{-1} = 2\max\left\{{\hat{c}_\mathcal{G}},{\tilde{c}_\mathcal{G}}\right\}$ the inequality
\[  -c_\mathcal{G}|S(U)|^2 \geq  (\DD \eta(U)-\DD \eta(\hat{U}))S(U) \quad \text{ in } \mathcal G,
 \]
 which implies the third condition.\\
Finally,
\[ \DD S(U)= \left(\begin{array}{cc}
0_{n\times n} & 0_{n\times nd} \\ & \\[-4mm]
  %-\blockdiag\left(\vec{m}_i\dfrac{\del M_i}{\del \rho_i}\right)
  \mathcal{A}(U)   & \DD_\vec{m}\vec{s}(U)
\end{array}\right),  \]
with  $\mathcal{A}(\hat{U}) = 0_{nd\times n}$. 
We obtain
\begin{align*}%\label{Q_U}
\DD S(\hat{U}) = \left(\begin{array}{cc}
0_{n\times n} & 0_{n\times nd} \\  0_{nd \times n} & \DD_\vec{m} \vec{s}(\hat{U})
\end{array}\right). 
\end{align*}
The lower right block of this matrix is invertible as shown above. Consequently,
\[ \ker(\DD S(\hat{U}))= \Span\{\vec{e}_1,...,\vec{e}_n\} \subset \R^{n(d+1)}. \]

Due to the zero block in $\DD F_{j}(\hat{U})$, the corresponding eigenvectors must have non-zero entries at the $n+1$-th to $n(d+1)$-th position. Therefore, the last condition of Theorem \ref{yongtheorem} holds.

We verified all the conditions of Theorem \ref{yongtheorem}. 
Hence, the system \eqref{Euler-MS} with $U_0$ as initial value has a unique solution $U=U(\vec{x},t) \in C([0,T),H^s(\R^d)),\, s\geq \lfloor d/2\rfloor +2$. The entropy inequality \eqref{entropyineq} is a consequence of \eqref{eq:compat}, \eqref{reform}, and $\DD\eta(\hat{U})S(U)=0$.
\end{proof}

\begin{rem}
\begin{enumerate}
\item Note that due to $s \geq \lfloor d/2 \rfloor +2$ we have with the Sobolev embedding theorem even $U \in C^1(\R^d\times (0, T))$.
\item The symmetry of $\DD^2\eta(U)\DD F_{j}(U)$ follows directly from the compatibility condition \eqref{eq:compat} of the entropy-entropy flux pair and the strict convexity of $\eta$. However, since the matrices are needed in the proof anyway, we checked this property by hand.
\end{enumerate}
\end{rem}

%%%%%%%%%%%%%%%%%%%%%%%%%%%%%%%%%%%%%%%%%%%%%%%%%%%%%%%%%%%%%%%%%%%%%%%%%%%%%%%%%%
\section{Convergence to  the Parabolic Limit System\label{sec:limit}}
%%%%%%%%%%%%%%%%%%%%%%%%%%%%%%%%%%%%%%%%%%%%%%%%%%%%%%%%%%%%%%%%%%%%%%%%%%%%%%%%%%%%
The goal of this section is to prove the convergence of solutions of \eqref{Euler-Darcy-MS-eps}, \eqref{Euler-Darcy-MS-eps_ini}   to solutions of an IVP 
for the parabolic limit system \eqref{Darcy-MS-limit} as $\eps$ tends to zero.
Due to Theorem \ref{thmregu} there exists for each $\eps >0$  a unique global solution
 $ {U}^\eps = (\vec{r}^{\eps\top},\vec{m}^{\eps\top})^T$ to the IVP for the 
$\eps$-scaled system \eqref{Euler-Darcy-MS-eps}. However, the convex set $\mathcal{G}$ might depend on $\eps$ such that the set of admissible initial conditions could shrink to the equilibrium for $\eps \to 0$. 
The techniques of \cite{Peng20161103} allow to show that there exists a time interval independent of $\eps$ with solutions
 $U^\eps%= (\vec{r}^{\eps\top},\vec{m}^{\eps\top})
 $ existing. Hence, we assume that there is a time $T>0$ and a compact, convex set $\mathcal{G}$ such that for all $\eps >0 $ 
 the solutions $U^\eps%= (\vec{r}^{\eps\top},\vec{m}^{\eps\top})
 $
  exist on the interval $(0,T)$ and are contained in $\mathcal{G}$.

Our convergence proof relies on the relative entropy method which goes
back to  \cite{Dafermos2,Dafermos1} and \cite{DiPerna}. 
This technique only requires one solution to be  a strong (in fact Lipschitz continuous) solution, whereas the other can be a discontinuous  entropy solution. We regard the solutions to \eqref{Euler-Darcy-MS-eps} as weak solutions and the solution of the limit system as strong solution.
Here we rely on the technical framework that 
has been established in \cite{Tzavaras}. We start to prove a dissipation relation (Proposition \ref{entprop}) for so-called relative entropies in 
Section \ref{relativeentropy} and conclude the convergence estimate with the main result in Theorem \ref{thm:paraboliclimit} of
 Section \ref{sec:convergence}.

In the following we omit again the arguments in $\lambda_{ij}$. With a slight misuse of notation  we use the expression 
$(\vec{r}^\eps,\vec{m}^\eps)$ for the solution $U^\eps$.

%%%%%%%%%%%%%%%%%%%%%%%%%%%%%%%%%%%%%%%%%%%%%%%%%%%%%%%%%%%%%%%%%%%
\subsection{The Relative Entropy Estimate\label{relativeentropy}}
%%%%%%%%%%%%%%%%%%%%%%%%%%%%%%%%%%%%%%%%%%%%%%%%%%%%%%%%%%%%%%%%%%%%%%%%
Let us consider \eqref{Euler-Darcy-MS-eps}, \eqref{Euler-Darcy-MS-eps_ini} for $\eps >0$.
To obtain a convergence estimate for the solutions $ (\vec{r}^\eps,\vec{m}^\eps) $ of \eqref{Euler-Darcy-MS-eps}, \eqref{Euler-Darcy-MS-eps_ini}  we  start to fix well prepared functions for 
the initial conditions in \eqref{Euler-Darcy-MS-eps_ini} and \eqref{Darcy-MS-limit_ini} on the entire $\R^d$. 
Let a number $R_0 > 0$ and $\hat{\vec{r}} \in (0,\infty)^n$  be given. We restrict the 
initial datum $(\vec{r}^\eps_0, \vec{m}^\eps_0) \in L^\infty(\R^d)$  in \eqref{Euler-Darcy-MS-eps_ini} to take values in $\mathcal G$ and to satisfy
\begin{align} \label{compactsupp} 
(\vec{r}^\eps_0(\vec{x}),\vec{m}^\eps_0(\vec{x})) &= (\hat{\vec{r}},0)  \text{ for } |\vec{x}|>R_0.
\end{align}
For the initial datum $\vec{\bar{r}}_0 \in C^3(\R^d)$ of the limit equation \eqref{Darcy-MS-limit} we impose the analogous condition 
\begin{align} \label{compsupppm}
\bar{\vec{r}}_0(\vec{x}) &= \hat{\vec{r}}, \ \text{ for } |\vec{x}|>R_0.
\end{align}

Using  the entropy-entropy flux pair $ (\eta, \vec{q})$ from   \eqref{eta}
we define an  \textit{entropy solution} $(\vec{r}^\eps,\vec{m}^\eps) \in L^\infty(\R^d \times (0,T))$ of \eqref{Euler-Darcy-MS-eps},
\eqref{Euler-Darcy-MS-eps_ini}  as a weak solution of \eqref{Euler-Darcy-MS-eps}, \eqref{Euler-Darcy-MS-eps_ini} 
that  takes values in $\mathcal G$  and satisfies
\begin{align}\label{weakentropyineq}
\del_t \eta(\vec{r}^\eps,\vec{m}^\eps)+\frac{1}{\eps}\div \vec{q}(\vec{r}^\eps,\vec{m}^\eps) +\frac{1}{\eps^2} \left(\sum_{i=1}^n M_i\frac{|\vec{m}^\eps_i|^2}{\rho^\eps_i}+ \frac{1}{2} \sum_{i,j=1}^n \lambda_{ij}\rho^\eps_i\rho^\eps_j \left|\frac{\vec{m}^\eps_i}{\rho^\eps_i}-\frac{\vec{m}^\eps_j}{\rho^\eps_j}\right|^2\right) \leq 0
\end{align}
in $D'(\R^d \times [0,T))$. Note that entropy flux scales with $\eps^{-1}$ according to the flux scaling in \eqref{Euler-Darcy-MS-eps}.

Further, let $\vec{\bar{r}} \in C^{3,1}(\R^d \times (0,T))\coloneqq \{\vec{g}\  | \ \vec{g}(\cdot,t) \in C^3(\R^d), t\in (0,T), \ \vec{g}(\vec{x},\cdot) \in C^1((0,T)), \vec{x}\in \R^d\}$  with $\vec{\bar{r}} \in (0,\infty)^n$     be a smooth solution of
\eqref{Darcy-MS-limit} \eqref{Darcy-MS-limit_ini}.  We observe that $\vec{\bar{r}}$ 
satisfies for all $\eps >0$  the expanded but  equivalent   system
\begin{align}\label{defbarsystem}
\begin{split}
\del_t\vec{\bar{r}} + \frac{1}{\eps} \div(\vec{\bar{m}}) &=0, \\
\vec{\bar{m}} &= -\eps (\mathcal{B}(\vec{\bar{r}}))^{-1}  \nabla\vec{p}(\vec{\bar{r}}).
\end{split}
\end{align}
Recall that the matrix $\mathcal B$ has been defined in \eqref{eq:defB}.
The regularity of $\vec{\bar{r}}$  implies $\vec{\bar{m}} \in C^{2,1}(\R^d \times (0,T))$ for the momentum. Note that $\vec{\bar{m}}$ depends on $\eps$
which is suppressed in the notation. 
The equivalent formulation \eqref{defbarsystem} of system \eqref{Darcy-MS-limit} induces for the evaluation of $\vec{\bar{m}}$  at zero time by 
\eqref{compsupppm}
the compatibility condition
\begin{equation}\label{compam}
%\vec{\bar{r}}(\cdot,0) &= \vec{\bar{r}}_0 \in C^3(\R^d), \\
\vec{\bar{m}}(\vec{x},0) = -\eps \mathcal{B}(\vec{\bar{r}}_0(\vec{x}))^{-1}  \nabla\vec{p}(\vec{\bar{r}}_0(\vec{x})) =0 \text{ for } |\vec{x}|>R_0.
\end{equation}
For this choice of $\vec{\bar{m}}$
we define  now the relative entropy expression
\begin{align}\label{def:relentropy}
\begin{split}
\eta(\vec{r}^\eps,\vec{m}^\eps|\vec{\bar{r}},\vec{\bar{m}})&:= \eta(\vec{r}^\eps,\vec{m}^\eps)-\eta(\vec{\bar{r}},\vec{\bar{m}})-\DD_\vec{r}\eta(\vec{\bar{r}},\vec{\bar{m}})\cdot(\vec{r}^\eps-\vec{\bar{r}}) -\DD_{\vec{m}} \eta(\vec{\bar{r}},\vec{\bar{m}}) \cdot (\vec{m}^\eps-\vec{\bar{m}}) \\
%&= \frac{1}{2} \sum_{i=1}^n \rho^\eps_i \left|\frac{\vec{m}^\eps_i}{\rho^\eps_i}- \frac{\bar{\vec{m}}_i}{\bar{\rho}_i} \right|^2+ 
%h(\vec{r}^\eps)- h(\vec{\bar{r}})-D_{\vec{r}}h(\vec{\bar{r}})\cdot(\vec{r}^\eps-\vec{\bar{r}}) \\
&= \frac{1}{2} \sum_{i=1}^n \rho^\eps_i \left|\frac{\vec{m}^\eps_i}{\rho^\eps_i}- \frac{\bar{\vec{m}}_i}{\bar{\rho}_i} \right|^2+ 
\sum_{i=1}^n  h_i(\rho_i^\eps|\bar{\rho}_i),
\end{split}
\end{align}
with 
\[
h_i(\vec{\rho}^\eps_i|{\bar{\rho}_i}) :=  h_i(\rho^\eps_i)- h_i(\bar{\rho}_i)-   h_i'({\bar{\rho}_i})({\rho}^\eps_i-{\bar{\rho}_i}).
\]

The relative entropy flux is defined by
\begin{align}\label{def:relentropyflux}
\vec{q}(\vec{r}^\eps,\vec{m}^\eps|\vec{\bar{r}},\vec{\bar{m}}) :=&\ \vec{q}(\vec{r}^\eps,\vec{m}^\eps)-\vec{q}(\vec{\bar{r}},\vec{\bar{m}})-(\DD_\vec{r}\eta(\vec{\bar{r}},\vec{\bar{m}})^\top \otimes \mathcal{I}_d)(\vec{m}^\eps-\vec{\bar{m}}) \nonumber \\ 
& \quad - (\mathcal{I}_d \otimes \DD_{\vec{m}_1} \eta(\vec{\bar{r}},\vec{\bar{m}})^\top, \ldots, \mathcal{I}_d \otimes \DD_{\vec{m}_n} \eta(\vec{\bar{r}},\vec{\bar{m}})^\top ) (F(\vec{r}^\eps,\vec{m}^\eps)-F(\vec{\bar{r}},\vec{\bar{m}})) \nonumber  \\
=&  \sum_{i=1}^n \bigg(  \frac{1}{2}\vec{m}^\eps_i \left| \frac{\vec{m}^\eps_i}{\rho^\eps_i}-\frac{\bar{\vec{m}}_i}{\bar{\rho}_i}\right|^2 + \rho_i(h_i^\prime(\rho^\eps_i)-h_i^\prime(\bar{\rho}_i))\left(\frac{\vec{m}^\eps_i}{\rho^\eps_i}-\frac{\bar{\vec{m}}_i}{\bar{\rho}_i}\right) \\
& \hspace*{1cm} + \frac{\bar{\vec{m}}_i}{\bar{\rho}_i}h_i(\rho^\eps_i|\bar{\rho_i}) \bigg),\nonumber
\end{align}
with  $F$ being a vectorial collection of the  momentum  fluxes  given by 
\begin{equation}\label{vectorflux}
F(\vec{r},\vec{m})= \left[\left(\frac{\vec{m}_i \otimes \vec{m}_i}{\rho_i}+p_i(\rho_i)(\vec{e}_1^\top,\vec{e}_2^\top,\dots,\vec{e}_d^\top)^\top\right)\right]_{i=1}^n \in \R^{nd^2}.
\end{equation}
In the last formula we made use of  the notation 
\[ \left[\vec{u}_i\right]_{i=1}^n := (\vec{u}_1^\top,\dots,\vec{u}_n^\top)^\top \in \R^{nm}, \quad \vec{u}_i \in \R^m, i=1,\dots,n, 
\]
which appears frequently in the sequel, where also Lemma \ref{lem:rechenregel} will be used often.\\

After artificially expanding the system \eqref{Darcy-MS-limit} to obtain \eqref{barsystem}, we are able to compare the solutions $(\vec{r}^\eps,\vec{m}^\eps)$ and $(\bar{\vec{r}},\bar{\vec{m}})$ of \eqref{Euler-Darcy-MS-eps} and \eqref{barsystem}, respectively. 
%%%%%%%%%%%%%%%%%%%%%%%%%%%%%%%%
\begin{prop}\label{entprop}
Let Assumption \ref{assumption} hold, let the pressure $p_i$ satisfy $\eqref{passum1}$, and let the initial functions $(\vec{r}^\eps_0, \vec{m}^\eps_0) \in L^\infty(\R^d)$ and $\vec{\bar{r}}_0 \in C^3(\R^d)$   satisfy   \eqref{compactsupp}, \eqref{compsupppm}. 

Consider for  $\eps >0$  an entropy solution  $(\vec{r}^\eps,\vec{m}^\eps) \in L^\infty(\R^d \times (0,T)) $  of  \eqref{Euler-Darcy-MS-eps}, \eqref{Euler-Darcy-MS-eps_ini}  and  a smooth solution $(\bar{\vec{r}},\bar{\vec{m}}) \in  C^{3,1}(\R^d \times [0,T) )    \times  C^{2,1}\R^d \times [0,T) )  $ of 
\eqref{Darcy-MS-limit},  \eqref{Darcy-MS-limit_ini}, supposed to take values in a convex, compact set  $\mathcal G \subset G$.\\  

Then we have the estimate
\begin{equation} \label{entprop_formula}
\begin{array}{rcl}
\lefteqn{\int_0^T\int_{\R^d} \eta(\vec{r}^\eps,\vec{m}^\eps|\vec{\bar{r}},\vec{\bar{m}}) \del_t \psi    +\frac{1}{\eps}  \vec{q}(\vec{r}^\eps,\vec{m}^\eps|\vec{\bar{r}},\vec{\bar{m}}) \cdot \nabla \psi  \dd \vec{x} \dd t}\\[2.5ex]
   &\geq& \displaystyle   -   \int_{\R^d}  \eta(\vec{r}^\eps_0,\vec{m}^\eps_0|\vec{\bar{r}}_0,\vec{\bar{m}}(\cdot, 0)) \psi(\cdot,0)  \dd \vec{x} \\[2.5ex] 
   && \displaystyle   +  \int_0^T\int_{\R^d}  \bigg(  \frac{1}{\eps^2}R_\eps(\vec{r}^\eps,\vec{m}^\eps,\vec{\bar{r}},\vec{\bar{m}}) + Q_\eps +E_\eps\bigg)
  \dd \vec{x} \dd t,  
\end{array}
\end{equation}
with

\begin{align}\label{eq:defQER}
\begin{split}
R_\eps(\vec{r}^\eps,\vec{m}^\eps,\vec{\bar{r}},\vec{\bar{m}}) &= R_{1,\eps}(\vec{r}^\eps,\vec{m}^\eps,\bar{\vec{r}},\bar{\vec{m}}) + R_{2,\eps}(\vec{r}^\eps,\vec{m}^\eps,\bar{\vec{r}},\bar{\vec{m}}),  \\
R_{1,\eps}(\vec{r}^\eps,\vec{m}^\eps,\bar{\vec{r}},\bar{\vec{m}}) &= \sum_{i=1}^n M_i\rho^\eps_i \left|\frac{\vec{m}^\eps_i}{\rho^\eps_i}-\frac{\bar{\vec{m}}_i}{\bar{\rho}_i}\right|^2,  \\
R_{2,\eps}(\vec{r}^\eps,\vec{m}^\eps,\bar{\vec{r}},\bar{\vec{m}}) &= \frac{1}{2}  \sum_{i,j=1}^n \lambda_{ij} \Bigg[ \rho^\eps_i\rho^\eps_j \left|\left(\frac{\vec{m}^\eps_i}{\rho^\eps_i}-\frac{\vec{m}^\eps_j}{\rho^\eps_j}\right)-\left(\frac{\bar{\vec{m}}_i}{\bar{\rho_i}}-\frac{\bar{\vec{m}_j}}{\bar{\rho}_j}\right)\right|^2   \\ 
&\qquad \qquad \qquad + \rho^\eps_i \left(\frac{\bar{\vec{m}}_i}{\bar{\rho}_i}-\frac{\bar{\vec{m}}_j}{\bar{\rho}_j}\right)\cdot\left(\frac{\vec{m}^\eps_i}{\rho^\eps_i}-\frac{\bar{\vec{m}}_i}{\bar{\rho}_i}\right)(\rho^\eps_j-\bar{\rho}_j)  \\
&\qquad \qquad \qquad \left.-\rho^\eps_j\left(\frac{\bar{\vec{m}}_i}{\bar{\rho}_i}-\frac{\bar{\vec{m}}_j}{\bar{\rho}_j}\right)\cdot\left(\frac{\vec{m}^\eps_j}{\rho^\eps_j}-\frac{\bar{\vec{m}}_j}{\bar{\rho}_j}\right)(\rho^\eps_i-\bar{\rho}_i)\right],  \\
Q_\eps(\vec{r}^\eps,\vec{m}^\eps,\vec{\bar{r}},\vec{\bar{m}}) &= \frac{1}{\eps} (\DD^2\eta(\bar{\vec{r}},\bar{\vec{m}}) \otimes \mathcal{I}_d)\nabla\begin{pmatrix}
\vec{\bar{r}}\\\vec{\bar{m}}
\end{pmatrix}\cdot \begin{pmatrix}
\vec{0}\\ F(\vec{r}^\eps,\vec{m}^\eps|\vec{\bar{r}},\vec{\bar{m}})
\end{pmatrix},
 \\
E_\eps(\vec{r}^\eps,\vec{m}^\eps,\vec{\bar{r}},\vec{\bar{m}})&= \bar{\vec{e}}_\eps\cdot\left[\frac{\rho^\eps_i}{\bar{\rho}_i}\left(\frac{\vec{m}^\eps_i}{\rho^\eps_i}-\frac{\bar{\vec{m}}_i}{\bar{\rho}_i}\right)\right]_{i=1}^n,\\
\bar{\vec{e}}_\eps= \bar{\vec{e}}_\eps(\vec{\bar{r}},\vec{\bar{m}}) &= \frac{1}{\eps} \left[\div\left(\frac{\bar{\vec{m}}_i \bar{\vec{m}}_i^\top}{\bar{\rho}_i}\right)\right]_{i=1}^n - \eps \del_t(\mathcal{B}(\vec{\bar{r}})^{-1}\nabla\vec{p}(\vec{\bar{r}})).
\end{split}
\end{align}

\end{prop}

Before we present the proof of Proposition \ref{entprop} we summarize two remarks on the scaling of the 
remainder terms  $Q_\eps$ and $E_\eps$ with respect to $\eps$ which will be needed in Section \ref{sec:convergence}.
\begin{rem}\ \label{remark}
\begin{enumerate}
\item The first factor of $Q_\eps$ depends only on $(\bar{\vec{r}}, \bar{\vec{m}})$. Although $\bar{\vec{m}}$ involves $\eps$  the factor
 is independent  of $\eps$, i.~e., % $Q_\eps = O(1)$:
\begin{align*}
\frac{1}{\eps}\left( (\DD^2_{\vec{r},\vec{m}}\eta(\bar{\vec{r}},\bar{\vec{m}})\otimes \mathcal{I}_d)\nabla\vec{\bar{r}}+(\DD^2_{\vec{m},\vec{m}}\eta(\bar{\vec{r}},\bar{\vec{m}})\otimes \mathcal{I}_d)\nabla\vec{\bar{m}}\right) &= \frac{1}{\eps}\left[\nabla\left(\frac{\bar{\vec{m}}_i}{\bar{\rho}_i}\right)\right]_{i=1}^n \\
&= - \nabla\left(\diag\left(\frac{1}{\bar{\rho}_i \otimes 1_d}\right)\mathcal{B}(\vec{\bar{r}})^{-1} \nabla\vec{p}(\vec{\bar{r}})\right)\\ &= O(1).
%- \left[\left(h_i^\prime(\bar{\rho}_i\right)_{xx}\right]_{i=1}^n = O(1).
\end{align*}
\item 
Recalling the definition of $\tilde{B}$ from \eqref{eq:defB} the smoothness of $\bar{\vec{r}}$ and $\vec{p}$ implies 
\begin{align}\label{eq:eOeps}
\begin{split}
\bar{\vec{e}}_\eps&= \frac{1}{\eps} \left[\div\left(\frac{\bar{\vec{m}}_i \bar{\vec{m}}_i^\top}{\bar{\rho}_i}\right)\right]_{i=1}^n  - \eps\del_t(\mathcal{B}(\vec{\bar{r}})^{-1}\nabla\vec{p}(\bar{\vec{r}}))  \\ 
&= \frac{1}{\eps}\left[\div\left(\diag\left(\frac{1}{\bar{\rho}_i \otimes 1_d}\right)\left[\bar{\vec{m}_i} \bar{\vec{m}}_i^\top\right]_{i=1}^n\right)\right] - \eps\del_t(\mathcal{B}(\vec{\bar{r}})^{-1}\nabla\vec{p}(\bar{\vec{r}}))  \\
&= \eps\left[\div\left(\diag\left(\frac{1}{\bar{\rho}_i \otimes 1_d}\right)\left[\mathcal{\tilde{B}}(\vec{\bar{r}})^{-1}\nabla p_i(\bar{\rho}_i)  \left(\mathcal{\tilde{B}}(\vec{\bar{r}})^{-1}\nabla p_i(\bar{\rho}_i)\right)^\top\right]_{i=1}^n\right)\right]  \\
&\quad  - \eps\del_t(\mathcal{B}(\vec{\bar{r}})^{-1}\nabla\vec{p}(\bar{\vec{r}}))  \\
&= O(\eps).
\end{split}
\end{align}
 Hence the vector $\bar{\vec{e}}_\eps=O(\eps)$ in $E_\eps$ is of order $O(\eps)$.
\end{enumerate}
\end{rem}

\textit{Proof (of Proposition \ref{entprop}).} To simplify notations we  may  omit  the index $\eps$ and write 
$  (\vec{r},\vec{m}) = (\vec{r}^\eps,\vec{m}^\eps)$.
The  entropy solution  $(\vec{r},\vec{m})$ of the IVP for \eqref{Euler-Darcy-MS-eps} satisfies the inequality \eqref{weakentropyineq}.

In order to derive a similar expression for  the solution $\vec{\bar{r}}$ of the IVP for \eqref{Darcy-MS-limit} 
we rewrite  the equivalent system \eqref{defbarsystem}  for the pairing $(\vec{\bar{r}},\vec{\bar{m}})$ further.  

With $\lambda_{ii}r_i = -\sum_{i\neq j} \lambda_{ij}r_j$,  \eqref{lambda}, and \eqref{lambdasum},   it is easy to see that the solution $(\vec{\bar{r}},\vec{\bar{m}})$ of \eqref{defbarsystem} also satisfies 
\begin{align}\label{barsystem}
\begin{split}
\del_t\vec{\bar{r}} + \frac{1}{\eps} \div(\vec{\bar{m}}) &= 0, \\
\del_t\vec{\bar{m}} + \frac{1}{\eps} \div(F(\vec{\bar{r}},\vec{\bar{m}})) &= \bigg[-\frac{1}{\eps^2} M_i\bar{\vec{m}}_i-\frac{1}{\eps^2}\sum_{j=1}^n \lambda_{ij}(\bar{\rho}_j\bar{\vec{m}}_i-\bar{\rho}_i\bar{\vec{m}}_j)\bigg]_{i=1}^n+\bar{\vec{e}}_\eps (\vec{\bar{r}},\vec{\bar{m}}),%\nonumber
\end{split}
\end{align}
with $\bar{\vec{e}}_\eps$ from \eqref{eq:defQER} and $F$ from \eqref{vectorflux}.\\
With \eqref{sec:convergence} we see that
$(\vec{\bar{r}},\vec{\bar{m}})$ satisfies in the sense of distributions 
\begin{align}\label{2.19}
\begin{split}
\del_t \eta(\vec{\bar{r}},\vec{\bar{m}}) +\frac{1}{\eps} \div \vec{q}(\vec{\bar{r}},\vec{\bar{m}}) 
%&= -\frac{1}{\eps^2} \nabla_\vec{m}\eta(\vec{\bar{r}},\vec{\bar{m}})\cdot\left[M_i\bar{m}_i+\sum_{j=1}^n \lambda_{ij}(\bar{\rho}_j\bar{m}_i-\bar{\rho}_i\bar{m}_j)\right]_{i=1}^n + \nabla_{\vec{m}}\eta(\vec{\bar{r}},\vec{\bar{m}})\cdot \bar{e}_\eps  \\
=& -\frac{1}{\eps^2}\left(\sum_{i=1}^n M_i \frac{|\bar{\vec{m}}_i|^2}{\bar{\rho}_i}+ \frac{1}{2}\sum_{i,j=1}^n \lambda_{ij}\bar{\rho}_i\bar{\rho}_j\left|\frac{\bar{\vec{m}}_i}{\bar{\rho}_i}-\frac{\bar{\vec{m}}_j}{\bar{\rho}_j}\right|^2\right)  \\
&+ \DD_{\vec{m}}\eta(\vec{\bar{r}},\vec{\bar{m}})\cdot \bar{\vec{e}}_\eps.
\end{split}
\end{align}
Before we use the entropy relations \eqref{weakentropyineq} and \eqref{2.19} we return to the weak formulations:
We subtract  the weak formulations  of   \eqref{barsystem}  from  the weak formulation for 
\eqref{Euler-Darcy-MS-eps} and obtain %in ${\mathcal D}'(\R^d\times [0,T))$
for the mass balance equations 
 \begin{align}\label{2.31}
 \raisetag{-5ex}&\hspace*{5ex}-\int_0^T  \int_{\R^d} \del_t\boldsymbol{\phi}\cdot(\vec{r}-\bar{\vec{r}}) 
 + \frac{1}{\eps} \nabla\boldsymbol{\phi}\cdot(\vec{m}-\bar{\vec{m}}) \dd \vec{x} \dd \tau - \int_{\R^d} \boldsymbol{\phi}(\vec{x},0)\cdot(\vec{r}_0-\bar{\vec{r}}_0)  \dd \vec{x} =0.
 \end{align}
 Using the definition  of  $\vec{\bar e}_\eps$     in \eqref{eq:defQER} yields for the momentum components 
 \begin{equation}
 \begin{array}{rcl}
 \lefteqn{-\int_0^T  \int_{\R^d}  \del_t\boldsymbol{\theta} \cdot (\vec{m}-\bar{\vec{m}})+\frac{1}{\eps}\nabla\boldsymbol{\theta} \cdot (F(\vec{r},\vec{m})-F(\bar{\vec{r}},\bar{\vec{m}})) \dd \vec{x}\dd\tau}\\[2.5ex]
 \lefteqn{-\int_{\R^d}  \boldsymbol{\theta}(\vec{x},0) \cdot (\vec{m}_0-\bar{\vec{m}} (\cdot,0)) \dd \vec{x}}\\
 &=& \displaystyle   \int_0^T  \int_{\R^d}  \boldsymbol{\theta} \cdot \Bigg(-\frac{1}{\eps^2}\bigg[M_i(\vec{m}_i-\bar{\vec{m}}_i) \bigg]_{i=1}^n \\[2.5ex]
 &&  \displaystyle {} \hspace*{1cm} +\bigg[ \sum_{j=1}^n \lambda_{ij} (\rho_j\vec{m}_i-\rho_i\vec{m}_j + \bar{\rho}_j \bar{\vec{m}}_i - \bar{\rho}_i \bar{\vec{m}}_j) \bigg]_{i=1}^n -\bar{\vec{e}}_\eps\Bigg) \dd \vec{x} \dd\tau. 
 \end{array}
 \label{2.32}
 \end{equation}
 Here, $\boldsymbol{\phi}$ and $\boldsymbol{\theta}$ are vector-valued test functions with compact support in  
 $\R^d \times [0,T)$.
 We make with some function $\psi \in C^\infty_0(\R^d \times [0,T))$ the specific choices
 \begin{align*}
 \boldsymbol{\phi}(\vec{x},\tau) &=  \psi (\vec{x},\tau)\DD_\vec{\vec{r}} {\eta}(\bar{\vec{r}}(\vec{x},\tau),\bar{\vec{m}}(\vec{x},\tau)),\\[1.2ex] \boldsymbol{\theta}(\vec{x},\tau)& =\psi (\vec{x},\tau)\DD_\vec{m} {\eta}(\bar{\vec{r}}(\vec{x},\tau),\bar{\vec{m}}(\vec{x},\tau)), 
 \end{align*}
 which lead  in \eqref{2.31} and \eqref{2.32} to 
 \begin{equation}
 \begin{array}{rcl} 
 \lefteqn{ \int_0^T\int_{\R^d} \left( \DD_\vec{r} {\eta}\left(\bar{\vec{r}},\bar{\vec{m}}\right)\cdot \left(\vec{r}-\bar{\vec{r}}\right)  + \DD_\vec{m} {\eta}\left(\bar{\vec{r}},\bar{\vec{m}}\right)\cdot \left(\vec{m}-\bar{\vec{m}}\right)\right)  \psi_t    \dd \vec{x}  \dd t}\\[2.5ex]
 \lefteqn{\displaystyle  + \int_0^T\int_{\R^d}  \frac{1}{\eps^2}  \Big( (\DD_\vec{r} {\eta}\left(\bar{\vec{r}},\bar{\vec{m}}\right) \otimes \nabla \psi) \cdot \left(\vec{m}-\bar{\vec{m}}\right)  + (\DD_\vec{m} {\eta}\left(\bar{\vec{r}},\bar{\vec{m}}\right)\otimes \nabla \psi) \cdot  (F(\vec{r},\vec{m})-F(\bar{\vec{r}},\bar{\vec{m}}))  
 	\Big) \dd \vec{x}  \dd t} \\ [2.5ex]     
& \hspace*{0.9cm}=& \displaystyle-   \int_{\R^d} \bigg( \DD_\vec{r} {\eta}\left(\bar{\vec{r}}_0,\bar{\vec{m}}_0\right)\cdot \left(\vec{r}_0-\bar{\vec{r}}_0\right) \\
&& {} \hspace*{2cm}\displaystyle + \DD_\vec{m} {\eta}\left(\bar{\vec{r}}_0,\bar{\vec{m}}(\cdot,0)\right)\cdot \left(\vec{m}_0-\bar{\vec{m}}(\cdot,0)\right)\bigg)  \psi(\cdot,0)  \dd \vec{x}\\[2.5ex] 
&& {}\displaystyle  -\int_0^T \int_{\R^d} J_\eps  \psi  \dd \vec{x}\dd t.
\end{array} \label{eq:comb3}     
 \end{equation}
 The term  $J_\eps$  in \eqref{eq:comb3}   is defined as

\begin{flalign}  %\label{2.21a}
J_\eps := 
 &{} \DD_\vec{m} \eta(\vec{\bar{r}},\vec{\bar{m}}) \nonumber \\
& \cdot \left( -\frac{1}{\eps^2}\left[ M_i(\vec{m}_i-\bar{\vec{m}}_i)+\sum_{j=1}^n \lambda_{ij}(\rho_j \vec{m}_i-\rho_i \vec{m}_j-\bar{\rho}_j \bar{\vec{m}}_i+\bar{\rho}_i \bar{\vec{m}}_j)\right]_{i=1}^n -\bar{\vec{e}}_\eps\right)  \nonumber \\
&+  {\del_t\big[\DD_\vec{r}\eta(\vec{\bar{r}},\vec{\bar{m}}) \big]}\cdot(\vec{r}-\vec{\bar{r}})+ {\del_t\big[\DD_\vec{m}\eta(\vec{\bar{r}},\vec{\bar{m}})\big]}\cdot (\vec{m}-\vec{\bar{m}}) \nonumber \\
&+\frac{1}{\eps} \nabla (\DD_\vec{r}\eta(\vec{r},\vec{m}))\cdot(\vec{m}-\vec{\bar{m}})+\frac{1}{\eps}\nabla(\DD_\vec{m}\eta(\vec{\bar{r}},\vec{\bar{m}})) \cdot (F(\vec{r},\vec{m})-F(\vec{\bar{r}},\vec{\bar{m}})). \nonumber 
\end{flalign}
%%%%%%%%%%%%%%%%%%%%%%%%%%%%
Combining the  entropy inequality \eqref{weakentropyineq} for $(\vec{r}^\eps, \vec{m}^\eps)$, the entropy equation \eqref{2.19}  for 
$(\bar{\vec{r}}, \vec{\vec{m}})$ and the relation \eqref{eq:comb3} 
the definition of the relative entropy-entropy flux pair in \eqref{def:relentropy},
\eqref{def:relentropyflux}  implies that the inequality
\begin{align*}
& \hspace*{-0.5cm}\del_t\eta(\vec{r},\vec{m}|\vec{\bar{r}},\vec{\bar{m}}) +\frac{1}{\eps} \div \vec{q}(\vec{r},\vec{m}| \vec{\bar{r}},\vec{\bar{m}}) \\[2ex]
\leq& -\frac{1}{\eps^2} \left( \DD_\vec{m}\eta(\vec{r},\vec{m})\cdot \bigg[M_i\vec{m}_i+\sum_{j=1}^n \lambda_{ij}(\rho_j\vec{m}_i-\rho_i\vec{m}_j)\bigg]_{i=1}^n\right) \\[2ex]
& +\frac{1}{\eps^2} \left( \DD_\vec{m} \eta(\vec{\bar{r}},\vec{\bar{m}}) \cdot \bigg[ M_i\bar{\vec{m}}_i + 
\sum_{j=1}^n \lambda_{ij}(\bar{\rho}_j\bar{\vec{m}}_i-\bar{\rho}_i\bar{\vec{m}}_j)\bigg]_{i=1}^n \right) \\[2ex]
&- \DD_{\vec{m}}\eta(\vec{\bar{r}},\vec{\bar{m}})\cdot \bar{\vec{e}}_\eps - \epsilon\eps 
\end{align*}
holds in the weak sense.\\
In the term  $J_\eps$   we use \eqref{2.19} and Lemma \ref{lem:rechenregel} to compute the time derivative of $ \nabla_{\vec{m}} \eta (\vec{\bar{r}},\vec{\bar{m}})$   by  the chain rule  which leads to 
\begin{flalign}%\label{2.21b}
J_\eps= &-\frac{1}{\eps^2} \DD_\vec{m} \eta(\vec{\bar{r}},\vec{\bar{m}}) \cdot \left[ M_i(\vec{m}_i-\bar{\vec{m}}_i)+\sum_{j=1}^n \lambda_{ij}(\rho_j \vec{m}_i-\rho_i \vec{m}_j-\bar{\rho}_j \bar{\vec{m}}_i+\bar{\rho}_i \bar{\vec{m}}_j)\right]_{i=1}^n \mkern-20mu  \nonumber  \\[2ex]
  &  -\DD_\vec{m} \eta(\vec{\bar{r}},\vec{\bar{m}}) \cdot \bar{\vec{e}}_\eps \nonumber &\\[2ex]
&+\DD^2 \eta(\vec{\bar{r}},\vec{\bar{m}}) \del_t\begin{pmatrix}
\vec{\bar{r}}\\ \vec{\bar{m}}
\end{pmatrix} \cdot \begin{pmatrix}
\vec{r}-\vec{\bar{r}}\\ \vec{m}-\vec{\bar{m}}
\end{pmatrix} \nonumber    \\[2ex]
&+ \frac{1}{\eps} (\DD^2 \eta(\vec{\bar{r}},\vec{\bar{m}}) \otimes \mathcal{I}_d) \nabla\begin{pmatrix}
\vec{\bar{r}}\\ \vec{\bar{m}}
\end{pmatrix} \cdot \begin{pmatrix}
\vec{m}-\vec{\bar{m}} \\ F(\vec{r},\vec{m})-F(\vec{\bar{r}},\vec{\bar{m}})
\end{pmatrix} \nonumber &\\[2ex]
= &-\frac{1}{\eps^2}\DD_\vec{m} \eta(\vec{\bar{r}},\vec{\bar{m}})\cdot \left[M_i(\vec{m}_i-\bar{\vec{m}}_i)+\sum_{j=1}^n \lambda_{ij} (\rho_j\vec{m}_i-\rho_i\vec{m}_j-\bar{\rho}_j\bar{\vec{m}}_i+\bar{\rho}_i\bar{\vec{m}}_j) \right]_{i=1}^n \mkern-20mu\nonumber \\[2ex]
& - \DD_\vec{m} \eta(\vec{\bar{r}},\vec{\bar{m}}) \cdot \bar{\vec{e}}_\eps \nonumber  \\[2ex]
& + \DD^2 \eta(\vec{\bar{r}},\vec{\bar{m}}) \begin{pmatrix}
\vec{0} \\ \displaystyle  \left[-\frac{1}{\eps^2}M_i\bar{\vec{m}}_i - \frac{1}{\eps^2}\sum_{j=1}^n \lambda_{ij} (\bar{\rho}_j\bar{\vec{m}}_i-\bar{\rho_i}\bar{\vec{m}}_j)\right]_{i=1}^n +\bar{\vec{e}}_\eps
\end{pmatrix}\cdot \begin{pmatrix}
\vec{r}-\vec{\bar{r}}\\ \vec{m}-\vec{\bar{m}}
\end{pmatrix} \nonumber &\\[2ex]
&+\frac{1}{\eps}(\DD^2 \eta(\vec{\bar{r}},\vec{\bar{m}})\otimes \mathcal{I}_d) 
\nabla \begin{pmatrix}
\vec{\bar{r}}\\\vec{\bar{m}}
\end{pmatrix} \cdot \begin{pmatrix}
\vec{0}\\ F(\vec{r},\vec{m}|\vec{\bar{r}},\vec{\bar{m}})
\end{pmatrix}. & \nonumber
\end{flalign}
Finally we proceed with this  expression for $J_\epsilon$ and  deduce
\begin{align*}
	&\hspace*{-0.5cm}\del_t\eta(\vec{r},\vec{m}|\vec{\bar{r}},\vec{\bar{m}}) +\frac{1}{\eps} \div \vec{q}(\vec{r},\vec{m}| \vec{\bar{r}},\vec{\bar{m}}) \\
= &-\frac{1}{\eps^2}\left( \DD_\vec{m} \eta(\vec{r},\vec{m}) \cdot \bigg[ M_i \vec{m}_i+ \sum_{j=1}^n \lambda_{ij} (\rho_j \vec{m}_i - \rho_i \vec{m}_j)\bigg]_{i=1}^n \right) 
\\ &+ \frac{1}{\eps^2} \left( \DD_{\vec{m}} \eta(\vec{\bar{r}},\vec{\bar{m}}) \cdot \bigg[ M_i\bar{\vec{m}}_i + \sum_{j=1}^n \lambda_{ij} (\bar{\rho}_j\bar{\vec{m}}_i-\bar{\rho}_i\bar{\vec{m}}_j) \bigg]_{i=1}^n\right) \\
&+\frac{1}{\eps^2} \DD_\vec{m} \eta(\vec{\bar{r}},\vec{\bar{m}}) \cdot \bigg[ M_i(\vec{m}_i-\bar{\vec{m}}_i) +
\sum_{j=1}^n \lambda_{ij}(\rho_j \vec{m}_i - \rho_i\vec{m}_j + \bar{\rho}_i \bar{\vec{m}}_j - \bar{\rho}_j \bar{\vec{m}}_i)\bigg]_{i=1}^n \\
& -\DD^2 \eta(\vec{\bar{r}},\vec{\bar{m}}) \begin{pmatrix}
\vec{0} \\ -\frac{1}{\eps^2}\left[ M_i\bar{\vec{m}}_i +  \sum_{j=1}^n \lambda_{ij} (\bar{\rho}_j\bar{\vec{m}}_i-\bar{\rho}_i\bar{\vec{m}}_j) \right]_{i=1}^n
\end{pmatrix}\cdot \begin{pmatrix}
\vec{r}-\vec{\bar{r}} \\ \vec{m}-\vec{\bar{m}}
\end{pmatrix} \\ 
&- \DD^2\eta(\vec{\bar{r}},\vec{\bar{m}})\begin{pmatrix}
\vec{0} \\ \bar{\vec{e}} 
\end{pmatrix} \cdot \begin{pmatrix}
\vec{r}-\vec{\bar{r}} \\ \vec{m}-\vec{\bar{m}}
\end{pmatrix} - \frac{1}{\eps} (\DD^2\eta(\vec{\bar{r}},\vec{\bar{m}})\otimes \mathcal{I}_d)\nabla\begin{pmatrix}
\vec{\bar{r}}\\ \vec{\bar{m}}
\end{pmatrix} \cdot \begin{pmatrix}
\vec{0} \\ F(\vec{r},\vec{m}|\vec{\bar{r}},\vec{\bar{m}})
\end{pmatrix} \\
=& -\frac{1}{\eps^2} R_\eps - Q_\eps - E_\eps.
\end{align*}
The last line follows from the definitions in \eqref{eq:defQER} and concludes the proof.
\hfill\qedhere

%%%%%%%%%%%%%%%%%%%%%%%%%%%%%%%%%%%%%%%%%%%%%%%%%%%%%%%%%%%%%%%%%%%
\subsection{The Convergence Estimate\label{sec:convergence}}
%%%%%%%%%%%%%%%%%%%%%%%%%%%%%%%%%%%%%%%%%%%%%%%%%%%%%%%%%%%%%%%%%%%
In this section we make an additional assumption on the pressure. Let there exist constants $a_i>0,\ i=1,\dots,n$, such that
\begin{align}
\label{passum1}
p_i^{\prime\prime}(r) &\leq a_i\dfrac{p_i^\prime(r)}{r} \quad \text{ for all } r > 0,  \text{ and } i=1,...,n.
\end{align}
The condition \eqref{passum1} is satisfied for e.g.~the  isentropic pressure laws 
 $p_i(r)=k_i r^{\gamma_i}$ ($\gamma_i \geq 1,\ k_i > 0$) with any choice of $a_i >0$.\\ 
Due to \eqref{passum1} we have 
\begin{align}\label{eq:passum2}
\frac{1}{a_i}p_i''(r) \leq h_i''(r) = \frac{p_i'(r)}{r}.
\end{align}
Note that $\ds p_i(\rho_i|\bar{\rho}_i) = p_i(\rho_i)-p_i(\bar{\rho}_i)-p_i'(\bar{\rho}_i)(\rho_i-\bar{\rho}_i) = (\rho_i-\bar{\rho}_i)^2 \int_0^1 \int_0^\tau p''(s\rho_i + (1-s) \bar{\rho}_i) \ddd s \ddd \tau.$  

Hence with $$|F(\vec{r}^\eps,\vec{m}^\eps|\vec{\bar{r}},\vec{\bar{m}})| = \eta(\vec{r}^\eps,\vec{m}^\eps|\vec{\bar{r}},\vec{\bar{m}}) + \sum_{i=1}^n p_i(\rho_i|\bar{\rho}_i)-h_i(\rho_i|\bar{\rho}_i)$$ 
the inequality \eqref{eq:passum2} implies with some $c>0$
\begin{align}
\label{eq:passumfol}
|F(\vec{r}^\eps,\vec{m}^\eps|\vec{\bar{r}},\vec{\bar{m}})| \leq c \eta(\vec{r}^\eps,\vec{m}^\eps|\vec{\bar{r}},\vec{\bar{m}}),
\end{align}
with $F(\vec{r},\vec{m}|\vec{\bar{r}},\vec{\bar{m}}) = F(\vec{r},\vec{m})-F(\vec{\bar{r}},\vec{\bar{m}})- \DD F(\vec{\bar{r}},\vec{\bar{m}}) (\vec{r}-\vec{\bar{r}},\vec{m}-\vec{\bar{m}}).$

We need to introduce a slightly different entropy-entropy flux pair to get a convergence estimate that corresponds to convergence in standard
Lebesgue spaces.  Subtracting the constant $\eta(\vec{\hat{r}}, \vec{0})= \sum_{i=1}^n h_i(\hat{\rho_i})$  from the entropy $\eta$ we obtain a modified entropy-entropy flux  pair 
$(\tilde{\eta}, \tilde{\vec{q}})$
with the property  $\tilde{\eta}(\hat{\vec{r}},0)=0$ by setting
\begin{align*}
\tilde{\eta}(\vec{r},\vec{m}) &= \eta(\vec{r},\vec{m}) -  \eta(\vec{\hat{r}}, \vec{0}) , \quad \tilde{\vec{q}}(\vec{r},\vec{m}) = \vec{q}(\vec{r},\vec{m}).
\end{align*}

Since \eqref{Euler-Darcy-MS-eps} is a hyperbolic balance law, 
due to \eqref{compactsupp} and the uniform bound in $\mathcal G$  the functions $(\vec{r}^\eps -\hat{\vec{r}},\vec{m}^\eps)$ 
have uniform compact support. 
Then again the uniform  boundedness  implies  that there 
  are constants $K_1,K_2 >0$  such that for any $\eps>0$ the entropy solution  $(\vec{r}^\eps,\vec{m}^\eps)$ of \eqref{Euler-Darcy-MS-eps}, \eqref{Euler-Darcy-MS-eps_ini} 
  satisfies
\begin{align}\label{eq:uniformbound}
\begin{split}
\max_{i=1,\dots,n}\sup_{t\in [0,T]} \left\{\int_{\R^d}|\rho_i^\eps(\vec{x},t)-\hat{\rho_i}(\vec{x},t)|\dd \vec{x}  \right\} \leq K_1, \\
\sup_{t\in [0,T]} \left\{\int_{\R^d} \tilde{\eta}(\vec{r}^\eps(\vec{x},t),\vec{m}^\eps(\vec{x},t)) \dd \vec{x}\right\} \leq K_2.
\end{split}
\end{align}
As discussed in the introduction to Section \ref{sec:limit} we will consider a classical  solution $\bar{\vec{r}}$  of   \eqref{Darcy-MS-limit}. 
Let $(\bar{\vec{r}} , \bar{\vec{m}}) \in C^{3,1}(\R^d \times [0,T) )    \times  C^{2,1}\R^d \times [0,T) )  $ be a  classical  solution of \eqref{Darcy-MS-limit}   (respectively   the equivalent system \eqref{defbarsystem}),  \eqref{Darcy-MS-limit_ini} with initial data satisfying \eqref{compsupppm}.  Since \eqref{Darcy-MS-limit} is a  regular parabolic system
 we can assume under corresponding conditions on $\bar{\vec{r}}_0$,that $(\bar{\vec{r}} , \bar{\vec{m}})$ is also contained in  
  $\mathcal G$. 
With the relative entropy $\tilde{\eta}(\vec{r}^\eps,\vec{m}^\eps|\bar{\vec{r}},\bar{\vec{m}}) = {\eta}(\vec{r}^\eps,\vec{m}^\eps|\bar{\vec{r}},\bar{\vec{m}}) -\sum_{i=1}^n h_i(\hat{\rho_i})  $,
we measure the distance between the solutions $(\vec{r}^\eps,\vec{m}^\eps)$  and $(\bar{\vec{r}} , \bar{\vec{m}}) $ via the expression
\begin{align}\label{yardstick} 
\varphi_\eps(t) := \int_{\R^d} \tilde{\eta}(\vec{r}^\eps(\vec{x},t),\vec{m}^\eps(\vec{x},t)|\bar{\vec{r}}(\vec{x},t) ,\bar{\vec{m}}(\vec{x},t) ) \dd \vec{x}.
\end{align}
Note that the conditions  \eqref{compactsupp}, \eqref{compsupppm}, and \eqref{compam} show that $\varphi_\eps(0)$ is finite. 
Due to the strict convexity of $\tilde{\eta}$  there are  some constants $c,C>0$ (which depend on $\mathcal{G}$) such that 
\begin{equation} \label{comparel}
c  |(\vec{s},\vec{n})-(\bar{\vec{s}},\bar{\vec{n}})|^2\leq \tilde{\eta}(\vec{s},\vec{n}|\bar{\vec{s}},\bar{\vec{n}}) \leq C |(\vec{s},\vec{n})-(\bar{\vec{s}},\bar{\vec{n}})|^2,\end{equation}
holds for all vectors  $ (\vec{s},\vec{n}), (\bar{\vec{s}},\bar{\vec{n}})  \in \mathcal{G}$. As a consequence 
 \eqref{yardstick} is compatible with  the $L^2$-difference of the solutions  $(\vec{r}^\eps,\vec{m}^\eps)$  and $(\bar{\vec{r}} , \bar{\vec{m}}) $.
We can now state the final theorem.

\begin{thm}[Asymptotic behavior for \eqref{Euler-Darcy-MS-eps}]\label{thm:paraboliclimit}
Let Assumption \ref{assumption} hold, let the pressure $p_i$ satisfy $\eqref{passum1}$, and let the initial functions $(\vec{r}^\eps_0, \vec{m}^\eps_0) \in L^\infty(\R^d)$ and $\vec{\bar{r}}_0 \in C^3(\R^d)$   satisfy   \eqref{compactsupp}, \eqref{compsupppm}. 
  
Consider for  $\eps >0$  an entropy solution  $(\vec{r}^\eps,\vec{m}^\eps) \in L^\infty(\R^d \times (0,T)) $  of  \eqref{Euler-Darcy-MS-eps}, \eqref{Euler-Darcy-MS-eps_ini}  and  a smooth solution $(\bar{\vec{r}},\bar{\vec{m}}) \in  C^{3,1}(\R^d \times [0,T) )    \times  C^{2,1}\R^d \times [0,T) )  $ of 
\eqref{Darcy-MS-limit},  \eqref{Darcy-MS-limit_ini}, supposed to take values in a convex, compact set  $\mathcal G \subset G$.\\  
Then there exist constants $c_i > 0, i=1,\dots,n$, such that for
\begin{align}\label{eq:weaklycoupled}
M_i \geq c_i \max_{j=1,\ldots,n}\max_{(r_i,r_j) \in  \mathcal{G}\times \mathcal{G}} \big\{|\lambda_{ij}(r_i,r_j)| \big\}
%, \text{ such that } R_1(\vec{r},\vec{m},\bar{\vec{r}},\bar{\vec{m}})/2 + R_2(\vec{r},\vec{m},\bar{\vec{r}},\bar{\vec{m}}) \geq 0.
\end{align}
we have the estimate
\begin{equation}\label{finalestimate}\varphi_\eps(t) \leq K (\varphi_\eps(0)+\eps^4) \qquad (t\in (0,T]).
\end{equation}
Here  $K>0$ is a constant that depends only on $T$, $\mathcal{G}$    and    $\bar{\vec{r}}$ but not on $\eps$. 
\end{thm}

\begin{rem}
\begin{enumerate}
\item If the initial datum $(\vec{r}^\eps_0, \vec{m}^\eps_0) $ converges  for $\eps \to 0$ to  $ (\bar{\vec{r}}_0, \bar{\vec{m}}(\cdot,0) )  $ 
 in  $L^2_{loc}(\R^d)$  the estimate  \eqref{finalestimate}   implies   
 \[   
    {\|(\vec{r}^\eps, \vec{m}^\eps)(\cdot,t)   -  (\bar{\vec{r}}, \bar{\vec{m}}  )(\cdot,t)  \|}_{L^2(\R^d)} \to 0,
 \]
   due to the  compatibility relation \eqref{comparel}.

\item 
The condition \eqref{eq:weaklycoupled} holds especially for $\lambda_{ij}=0, \ i,j=1,\dots,n$ what corresponds exactly to \cite{Tzavaras}. In gaseous mixtures \eqref{eq:weaklycoupled} is expected to hold, see Remark \ref{rem:model}.
\end{enumerate}
\end{rem}

\textit{Proof (of Theorem  \ref{thm:paraboliclimit}).} For the proof we write again $(\vec{r},\vec{m})=(\vec{r}^\eps,\vec{m}^\eps)$.

We consider the  relative entropy statement from \eqref{entprop_formula} in Proposition \ref{entprop} which holds also for the entropy pair $(\tilde {\eta}, \tilde{\vec{q}} )$. As test function $\psi $  we make the choice  $\psi(\vec{x},t) = \theta_\kappa(\tau)\omega_R(\vec{x}) $ with 
$\theta_\kappa$  given for $\kappa >0$  by
\begin{align}\label{eq:deftheta}
\theta_\kappa(\tau) := \begin{cases}
1, & 0\leq\tau <t,  \\
\frac{t-\tau}{\kappa}+1, & t \leq \tau < t+\kappa,\\
0, & \tau \geq t+\kappa,
\end{cases}
\end{align} 
and with $\omega_R$  given for $R,\delta>0$  by
\begin{align*}
\omega_R(\vec{x}) = \begin{cases}
1, & |\vec{x}|<R, \\
1+\frac{R-|\vec{x}|}{\delta}, & R < |\vec{x}| < R+\delta, \\
0, & \text{ else.}
\end{cases}
\end{align*}
By taking the limit $R \to \infty$, using the asymptotic properties \eqref{compactsupp}, \eqref{compsupppm} of $(\vec{r},\vec{m})$ and $(\bar{\vec{r}},\bar{\vec{m}})$, and finally sending $\kappa \to 0$ we obtain
using the definition of $\phi_\eps$ from \eqref{yardstick} the inequality
\begin{align}\label{2.34}
\varphi_\eps(t) + \frac{1}{\eps^2}\int_0^t\int_{\R^d} R_\eps(\vec{r},\vec{m},\bar{\vec{r}},\bar{\vec{m}})\dd \vec{x} \dd \tau \leq \varphi_\eps(0) +\int_0^t \int_{\R^d} \big(|Q_\eps|+|E_\eps| \big) \dd \vec{x} \dd\tau,
\end{align}
with $Q_\eps,E_\eps,$ and $R_\eps$ from \eqref{eq:defQER} in Proposition \ref{entprop}. 

Due to Remark \ref{remark} and \eqref{eq:passumfol} it holds
\begin{align*}
\int_0^t\int_{\R^d} |Q_\eps| \dd\vec{x}\dd\tau \leq C_1 \int_0^t \varphi_\eps(\tau)\dd\tau,
\end{align*}
where $C_1>0$ depends on the $L^\infty$-norm of $\nabla \bar{\vec{r}}$ but not on $\eps$.
%$\left\|\div\left( \diag\left(\dfrac{1}{\bar{\rho}_i\otimes \mathcal{I}_d}\right)\mathcal{B}^{-1}(\vec{\bar{r}})\nabla\vec{p}(\vec{\bar{r}})\right)\right\|_{L^\infty}. $
The error term $E_\eps$ can be estimated for any number $C_2 >0$  with Young's inequality by
\[
\begin{array}{rcl}
\displaystyle \int_0^t \int_{\R^d} |E_\eps| \dd \vec{x}\dd\tau &\leq&  \displaystyle \frac{C_2\eps^2}{2} \int_0^t\int_{\R^d} \sum_{i=1}^n \left|\frac{\bar{\vec{e}}_{\eps,i}}{\bar{\rho}_i}\right|^2\rho_i \dd \vec{x} \dd\tau  \\
&&\displaystyle  + \frac{1}{2C_2\eps^2}\int_0^t\int_{\R^d} \sum_{i=1}^n M_i\rho_i \left|\frac{\vec{m}_i}{\rho_i}-\frac{\bar{\vec{m}}_i}{\bar{\rho}_i}\right|^2 \dd \vec{x}\dd\tau.
\end{array}
\]
Additionally, we have from \eqref{eq:eOeps} with $\bar{\vec{e}}_\eps = O(\eps)$   (see Remark \ref{remark}) the inequality
\begin{align*}
\int_0^t\int_{\R^d} \sum_{i=1}^n \left|\frac{\bar{\vec{e}}_{\eps,i}}{\bar{\rho}_i}\right|^2 \rho_i \dd \vec{x}\dd\tau &\leq \sum_{i=1}^n\left(\left\|\frac{\bar{\vec{e}}_{\eps,i}}{\bar{\rho}_i}\right\|^2_{L^\infty} \int_0^t \int_{\R^d}\left|\rho_i-\hat{\rho}_{i}\right|\dd \vec{x}\dd\tau \right.\\
&\qquad  \qquad \qquad \qquad   \left.+ |\hat{\rho}_{i}|\int_0^t\int_{\R^d} \left|\frac{\bar{\vec{e}}_{\eps,i}}{\bar{\rho}_i}\right|^2 \dd \vec{x} \dd\tau\right) \\
&\leq C_3\eps^2t,
\end{align*}
where the constant $C_3>0$ depends on $T,K_1$ from \eqref{eq:uniformbound}, $\mathcal{G}$, and also on $\bar{\vec{r}}$ through \eqref{eq:eOeps}.

Plugging these estimates into \eqref{2.34} leads to 
\begin{align*}
\hspace*{-1.5cm}\varphi_\eps(t) + \frac{1}{\eps^2}\int_0^t\int_{\R^d} R_\eps(\vec{r},\vec{m},\bar{\vec{r}},\bar{\vec{m}})-\frac{1}{2C_2}R_{1,\eps}(\vec{r},\vec{m},\bar{\vec{r}},\bar{\vec{m}})\dd \vec{x}\dd\tau \\ \leq \varphi_\eps(0) + C_1 \int_0^t \varphi_\eps(\tau) \dd\tau + C_3 \eps^4 t.
\end{align*}

We need the integral on the left hand side  of the last estimate to be positive. The integrand reads as 
$$ R_\eps(\vec{r},\vec{m},\bar{\vec{r}},\bar{\vec{m}})-\frac{1}{2C_2}R_{1,\eps}(\vec{r},\vec{m},\bar{\vec{r}},\bar{\vec{m}}) = \left(1-\frac{1}{2C_2}\right) R_{1,\eps} + R_{2,\eps}.$$
The term $R_{1,\eps}$ is positive and scales with the mobilities $M_i$, whereas the term $R_{2,\eps}$ can have arbitrary sign and scales with the diffusion coefficients $\lambda_{ij}$.

Hence, if the first term dominates, we can assure positivity of the integral. This follows with \eqref{eq:defQER}, \eqref{eq:weaklycoupled}
and choosing $C_2$ sufficiently large.  Then  Gronwall's inequality yields a constant $K>0$ such that 
\[\varphi_\eps(t) \leq K(\varphi_\eps(0)+\eps^4), \quad t\in(0,T].\]
\hfill\qedhere

\section{Conclusions}\label{sec:conclusions}
We have presented how to derive the system \eqref{Euler-MS} in such a way that it automatically satisfies an entropy inequality and hence fulfills the second law of thermodynamics. There exist smooth solutions globally in time to this system if the smooth initial data are close enough to an equilibrium. In an asymptotic time regime we show the convergence to a parabolic limit system generalizing results on the single-component case.

\section*{Acknowledgments}
The authors kindly acknowledge the financial support of this work by the Deutsche Forschungsgemeinschaft (DFG)
in the frame of the International Research Training Group "Droplet Interaction Technologies" (DROPIT).

\appendix
\renewcommand{\theequation}{\Alph{section}.\arabic{equation}}
\section{Differential Operators and Matrix Algebra\label{sec:appA}}
We collect some definitions from vector analysis and matrix algebra which are frequently used in Sections \ref{sec:mod}--\ref{sec:limit}.\\
For some  vector $\vec{u}(\vec{x}) = (u_1(\vec{x}),\dots,u_n(\vec{x}))^\top \in \R^n$ the (generalized) gradient a is defined as 
\begin{equation}\label{gengrad}
\nabla \vec{u}(\vec{x}) \coloneqq (\nabla u_1(\vec{x}),\dots \nabla u_n(\vec{x}))^\top \in \R^{nd}, 
\end{equation}
and for $\vec{v}(\vec{x}) = (\vec{v}_1^\top(\vec{x}),\dots, \vec{v}_n^\top(\vec{x}))^\top \in \R^{nd}$ the (generalized) divergence is given by
\begin{equation}
\div \vec{v}(\vec{x})\coloneqq \sum_{i=1}^n \div(\vec{v}_i(\vec{x})).
\end{equation}

By  $\otimes$  we denote the Kronecker product of two matrices, i.e., 
with $A\in \R^{m\times n}, B\in \R^{p \times q}$
\begin{align} A\otimes B &:= \begin{pmatrix}
a_{11} B & \dots & a_{1n} B \\
\vdots & \ddots & \vdots \\
a_{m1} B & \cdots & a_{mn} B
\end{pmatrix} \in \R^{mp\times nq}, \label{def:kronecker}\\
\intertext{and by $\odot$ the entrywise product for matrices of identical dimensions.
	We define $\blockdiag(\vec{x}_i),$ with $\vec{x}_i \in \R^d, i=1,\dots,n$, as }
\blockdiag(\vec{x}_i) &:=\begin{pmatrix}
\vec{x}_1 & 0_{d\times 1} & \cdots & \cdots &0_{d\times 1} \\
0_{d\times 1} & \vec{x}_2 & 0_{d\times 1} & \cdots & 0_{d\times 1} \\
\vdots & \ddots & \ddots & \ddots & \vdots\\
\vdots & \ddots & \ddots & \ddots & 0_{d\times 1} \\
0_{d\times 1} & \cdots & \cdots & 0_{d\times 1} & \vec{x}_n
\end{pmatrix}\in \R^{nd\times n}. \nonumber \\
\intertext{In addition, with matrices $A_i\in \R^{d\times d}, \ i=1,\dots,n$, let}
\blockdiag(A_i) &:=\begin{pmatrix}
A_1 & 0_{d\times d} & \cdots & 0_{d\times d}\\
0_{d\times d} & \ddots & \ddots & \vdots\\
\vdots & \ddots & \ddots & 0_{d\times d}\\
0_{d\times d} & \cdots & 0_{d\times d} & A_n
\end{pmatrix} \in \R^{nd\times nd}. \nonumber
\end{align}

We conclude with  the following rules  for the generalized gradient defined in \eqref{gengrad}. 
\begin{lem}\label{lem:rechenregel}
For smooth functions $\alpha \colon \R^d \to \R, \vec{a} \colon \R^d \to \R^{n}, $ $\vec{b} \colon \R^d \to \R^n, $ and $\vec{c} \colon \R^n \to \R^{n}$ it holds 
\begin{align*}
\nabla(\alpha(\vec{x})\vec{a}(\vec{x})) &= \vec{a}(\vec{x}) \otimes \nabla \alpha(\vec{x}) + \alpha(\vec{x}) \nabla \vec{a}(\vec{x}), \\
\nabla(\vec{c}(\vec{b}(\vec{x})) &= (\DD_\vec{b}\vec{c}(\vec{b}(\vec{x}))\otimes \mathcal{I}_d) \nabla \vec{b}(\vec{x}).
\end{align*}
\end{lem}
\begin{proof}
We have \begin{align*}
\nabla(\alpha(\vec{x})\vec{a}(\vec{x})) &= \begin{pmatrix}
\nabla (\alpha(\vec{x})a_1(\vec{x})) \\
\vdots\\
\nabla (\alpha(\vec{x})a_n(\vec{x})) \\
\end{pmatrix} = \begin{pmatrix}
\nabla (\alpha(\vec{x})) a_1(\vec{x}) + \alpha(\vec{x})\nabla a_1(\vec{x}) \\
\vdots \\
\nabla (\alpha(\vec{x})) a_n(\vec{x}) + \alpha(\vec{x})\nabla a_n(\vec{x}) \\
\end{pmatrix} \\[2ex]&= \vec{a}\otimes\nabla \alpha(\vec{x}) + \alpha(\vec{x}) \nabla \vec{a}(\vec{x})
\end{align*}
and
\begin{align*}
\nabla(\vec{c}(\vec{b}(\vec{x})) &= \begin{pmatrix}
\nabla c_1(\vec{b}(\vec{x})) \\
\vdots \\
\nabla c_n(\vec{b}(\vec{x})) 
\end{pmatrix}
= \begin{pmatrix}
\sum_{i=1}^n\DD_{b_i} c_1(\vec{b}(\vec{x}))\nabla b_i(\vec{x})   \\[2ex]
\vdots \\
\sum_{i=1}^n\DD_{b_i} c_n(\vec{b}(\vec{x}))\nabla b_i(\vec{x})
\end{pmatrix} \\[2ex]
&= \begin{pmatrix}
D_{b_1}c_1(\vec{b}(\vec{x})) \mathcal{I}_d &\ldots& D_{b_n}c_1(\vec{b}(\vec{x}))\mathcal{I}_d \\
\vdots& & \vdots \\
D_{b_1}c_n(\vec{b}(\vec{x}))\mathcal{I}_d &\ldots &D_{b_n}c_n(\vec{b}(\vec{x}))\mathcal{I}_d
\end{pmatrix} \nabla \vec{b}(\vec{x})\\[2ex]
 &= (D_\vec{b}\vec{c}(\vec{b}(\vec{x})) \otimes  \mathcal{I}_d) \nabla \vec{b}(\vec{x}).
\end{align*}
\end{proof}

\section{Global Classical Well-posedness of  IVPs for Hyperbolic Balance Laws\label{sec:appB}}

Let the state space $G \subset \R^{n(d+1)}$ be open and denote by $U: \R^d \times [0,T) \to G$  the  function that satisfies 
for some $T>0 $ the IVP for the nonlinear system of balance laws given by 
\begin{align}
	\del_tU + \sum_{\alpha=1}^d \del_{x_\alpha}F_\alpha(U)&= S(U)  \text{ in } \R^d\times (0,T). \label{yong1}
\end{align}
Here  $S : G \to \R^{n(d+1)}$ and $F_\alpha: G \to \R^{n(d+1)}, \ \alpha=1,\dots,d$ are  smooth functions with
\[ S(U) = \begin{pmatrix}
\vec{0} \\ \vec{s}(U)
\end{pmatrix}. \]
From now on we assume that $U$ can be split  according to $U=(\vec{r}^\top,\vec{m}^\top)^\top,$ with $\vec{r} \in \R^n, \vec{m}\in \R^{nd}$.
%, and do not distinguish the two forms.  
The system \eqref{yong1} can then be written as
\begin{align}\label{yong2}
	\del_t\begin{pmatrix}
		\vec{r}\\\vec{m}
	\end{pmatrix} + \sum_{\alpha=1}^d \del_{x_\alpha}F_\alpha(\vec{r},\vec{m})&= \begin{pmatrix} \vec{0} \\ \vec{s}(\vec{r},\vec{m}) \end{pmatrix}.
\end{align}

\begin{thm}[\cite{yong}]\label{yongtheorem}
	Let $s \geq s_0+1 = \lfloor d/2\rfloor +2$ be an integer and $\hat{U} \in G$ be a constant equilibrium state such that 
	 the following conditions hold:
	\begin{enumerate}
		\item The Jacobian $\DD_\vec{m} \vec{s}(\hat{U}) \in \R^{nd \times nd}$ is invertible.
		\item There exists a strictly convex smooth entropy function $\eta: \mathcal{G} \to \R$, defined in a convex, compact neighborhood $\mathcal{G} \subset G$
		 of $\hat{U}$, such that $\DD^2\eta(U)\DD F_{\alpha}(U)$ is symmetric for all $U\in \mathcal{G}$ and all $\alpha=1,\dots,d$.
		\item There is a positive constant $c_\mathcal{G}$ such that for all $U\in \mathcal{G}$,
		\[ [\DD\eta(U)-\DD\eta(\hat{U})]S(U) \leq -c_\mathcal{G} |S(U)|^2, \]
		where $|\cdot|$ denotes the  Euclidean norm of a vector. 
		\item The kernel $ \ker(\DD S(\hat{U}))$ of the Jacobian $\DD S(\hat{U})$ contains no eigenvector of the matrix \\$\sum_{\alpha=1}^d \omega_\alpha \DD F_{\alpha}(\hat{U})$, for any $\omega = (\omega_1,\dots,\omega_d) \in \mathbb{S}^{d-1}$.
	\end{enumerate}
	Then there is a constant $c_1>0$ such that for  $U_0\in H^s(\R^d)$ with  
	\[ \|U_0-\hat{U}\|_s \leq c_1 \]
	the  system of balance laws \eqref{yong2} with $U_0$ as its initial value has a unique global solution $U=U(\vec{x},t) \in C([0,T);H^s(\R^d))$ satisfying
	\begin{align*}
		\|U(\cdot,T)-\hat{U}\|_s^2+ \int_0^T\|S(U)(\cdot,t)\|_s^2 \dt + \int_0^T \|\nabla U(\cdot,t)\|^2_{s-1} \dt \leq c_2 \|U_0-\hat{U}\|_s^2
	\end{align*}
	for any $T>0$ and some $c_2>0$.
\end{thm}

\bibliographystyle{abbrv}
\bibliography{Literatur}

\end{document}